\documentclass[10pt,a4paper,oneside]{article}
\usepackage[utf8]{inputenc}
\usepackage[english]{babel}
\pdfoutput=1

\usepackage{amsmath,amsfonts,amssymb,amsopn,amscd,amsthm,mathtools}
\allowdisplaybreaks

\usepackage[yyyymmdd,hhmmss]{datetime}

\usepackage{fancyhdr}
\usepackage{tikz}
\newcommand{\FigureIntensityCone}[1]{
\begin{tikzpicture}[scale=#1]

\coordinate (zero) at (0,0);

\newcommand{\Size}{4}

\draw[->] (zero) -- (30:\Size);
\draw (zero) ++(0.5,0) ++ (30:1.6) node[rotate=30] {atomic};
\draw[->] (zero) -- (150:\Size);
\draw (zero) ++(-0.5,0) ++ (150:1.6) node[rotate=-30] {diffuse};

\fill (zero) circle (0.05cm);
\draw (0.2,-0.1) node[anchor=west] {0};

\draw (0.0,0.5) node {LLLs};

\newcommand{\Level}{1}
\draw plot [smooth] coordinates
 { (30:1.2*\Level)  (60:1.1*\Level)  (90:1.0*\Level) (150:0.8*\Level) };
\draw plot [smooth] coordinates
 { (150:1.2*\Level) (120:1.1*\Level) (90:1.0*\Level) (30:0.8*\Level) };

\draw (0,1.8) node {positive phase};
\draw (0,1.4) node {$\MeasuresPhasePositive=\MeasuresShearerStrict\subsetneq\MeasuresShearerExists$};

\renewcommand{\Level}{2.3}
\draw[dashed] (0,0) ++(30:\Level) arc (30:150:\Level);

\draw (-1.3,2.4) node {$\MeasuresPhaseZero$};
\draw (-1.3,2.8) node {zero phase};

\renewcommand{\Level}{3}
\draw (0,0) ++(30:\Level) arc (30:90:\Level);

\draw (2.0,3.2) node {no $1$-dependent PPs};
\draw (2.0,2.8) node {$\MeasuresEmpty$};

\draw (0,\Level) -- (0,\Size-0.25);

\end{tikzpicture}
}

\newcommand{\MathEnvOwn}{
\newtheorem{Thm}{Theorem}
\newtheorem{Cor}[Thm]{Corollary}
\newtheorem{Prop}[Thm]{Proposition}
\newtheorem{Lem}[Thm]{Lemma}
\newtheorem{Def}[Thm]{Definition}
}

\MathEnvOwn

\newcommand{\Naturals}{\mathbb{N}}
\newcommand{\NonNegInts}{\mathbb{N}_0}
\newcommand{\Integers}{\mathbb{Z}}
\newcommand{\Reals}{\mathbb{R}}

\newcommand{\Then}{\,\Rightarrow\,}
\newcommand{\Iff}{\,\Leftrightarrow\,}

\newcommand{\EnumSet}[1]{{\{{#1}\}}}
\newcommand{\DescSet}[2]{\{{#1}\mid {#2}\}}

\newcommand{\BigDescSet}[2]{{\left\{{#1}\,\!\middle|\,\!{#2}\right\}}}
\newcommand{\Cardinality}[1]{|{#1}|}
\newcommand{\IntRange}[1]{[{#1}]}
\newcommand{\Pochhammer}[2]{{#1^{[#2]}}}

\newcommand{\Lebesgue}{\mathcal{L}}
\newcommand{\D}[1]{\mathrm{d}{#1}}

\newcommand{\Modulus}[1]{|{#1}|}
\newcommand{\BigModulus}[1]{\left|{#1}\right|}

\newcommand{\ClosureOf}[1]{{\overline{#1}}}
\DeclareMathOperator{\diam}{diam}

\newcommand{\Proba}{\mathbb{P}}
\newcommand{\Expect}{\mathbb{E}}

\newcommand{\Space}{{\mathcal{X}}}
\newcommand{\Metric}{{\delta}}
\newcommand{\Distance}[2]{\Metric(#1,#2)}

\newcommand{\UnitSphere}[1]{U(#1)}
\newcommand{\ClosedUnions}{\mathcal{H}}

\newcommand{\BorelOf}[1]{\mathcal{B}^{#1}}
\newcommand{\BorelBounded}{\BorelOf{b}}
\newcommand{\BorelAll}{\BorelOf{}}
\newcommand{\BorelUnitDiam}{\BorelOf{1}}

\newcommand{\UPartSymbol}{\kappa}
\newcommand{\UPartNumOf}[1]{\UPartSymbol(#1)}
\newcommand{\UPartSup}{K}
\newcommand{\BorelFinite}{\BorelOf{\UPartSymbol}}

\newcommand{\MeasuresOf}[1]{\mathcal{M}^{#1}}
\newcommand{\MeasuresBounded}{\MeasuresOf{b}}

\newcommand{\DiffOpSymbol}{\Delta}

\newcommand{\MeasuresEmpty}{\MeasuresOf{\emptyset}}
\newcommand{\MeasuresPhaseZero}{\MeasuresOf{0}}
\newcommand{\MeasuresPhasePositive}{\MeasuresOf{+}}

\newcommand{\IsHCTuple}[1]{h_{#1}}

\newcommand{\OneDepPPM}{\mathcal{C}(M)}
\newcommand{\SupportPPOf}[1]{{#1^{\bullet}}}

\newcommand{\Atoms}{\mathcal{A}}
\newcommand{\DiffuseDomain}{\mathcal{D}}

\newcommand{\MeasuresShearerExists}{\MeasuresOf{\text{sh}}}
\newcommand{\MeasuresShearerStrict}{\MeasuresOf{>}}

\newcommand{\DistinguishedMinimalPP}{\mu}
\newcommand{\HardSphereM}[1]{{h_{#1,M}}}

\newcommand{\ShearersPP}[1]{\eta_{#1}}
\newcommand{\ShearersPPM}{\ShearersPP{M}}

\newcommand{\Penrose}[1]{\mathbf{P}(#1)}

\newcommand{\GfSymbol}{Z}                  
\newcommand{\GfG}[2]{\GfSymbol({#1},{#2})} 
\newcommand{\GfM}[1]{\GfSymbol({#1})}      
\newcommand{\GfRSymbol}{z}                 
\newcommand{\GfR}[3]{\GfRSymbol({#1},{#2},{#3})} 
\newcommand{\GfMR}[2]{\GfRSymbol(#1,#2)}   

\newcommand{\FunctionalM}[1]{f_{#1}}
\newcommand{\RootM}[1]{\lambda_{#1}}

\newcommand{\AvFun}{Q}
\newcommand{\AvProba}[1]{\AvFun(#1)}
\newcommand{\AvProbaRatio}[2]{q(#1,#2)}

\newcommand{\MNCVISets}{\mathcal{N}_{k,n}}

\newcommand{\ArxivOnly}[1]{#1}
\newcommand{\SelfOnly}[1]{}
\newcommand{\Abstract}{
\begin{abstract}
A point process is R-dependent, if it behaves independently beyond the minimum distance R.
This work investigates uniform positive lower bounds on the avoidance functions of R-dependent simple point processes with a common intensity.
Intensities with such bounds are described by the existence of Shearer's point process, the unique R-dependent and R-hard-core point process with a given intensity.
This work also presents several extensions of the Lovász Local Lemma, a sufficient condition on the intensity and R to guarantee the existence of Shearer's point process and exponential lower bounds.
Shearer's point process shares combinatorial structure with the hard-sphere model with radius R, the unique R-hard-core Markov point process.
Bounds from the Lovász Local Lemma convert into lower bounds on the radius of convergence of a high-temperature cluster expansion of the hard-sphere model.
This recovers a classic result of Ruelle on the uniqueness of the Gibbs measure of the hard-sphere model via an inductive approach à la Dobrushin.
\end{abstract}
}

\newcommand{\TitleFull}{Shearer's point process, the hard-sphere model and a continuum Lovász Local Lemma}

\newcommand{\TitleShort}{Shearer's PP, hard-spheres \& a continuum LLL}

\newcommand{\AuthorsFull}{Hofer-Temmel Christoph (math@temmel.me)}

\newcommand{\AuthorsShort}{Hofer-Temmel}

\newcommand{\Postal}%
{c/o FMW, MPC 10A, Postbus 10000, 1780 CA Den Helder, The Netherlands}

\newcommand{\KeyWords}%
{ avoidance function
; Lovász Local Lemma
; hard-sphere model
; partition function
; R-dependent point process
; R-hard-core
}
\newcommand{\KeyWordLine}{Keywords: \KeyWords{}}

\newcommand{\Head}{
 \maketitle
 \Abstract{}
 \KeyWordLine{}\\\MSC{}
}

\newcommand{\Acknowledgements}{
The author acknowledges the support of the VIDI project ``Phase transitions, Euclidean fields and random fractals'', NWO 639.032.916, while working at the VU Amsterdam.
The author thanks Marie-Colette van Lieshout for discussions about general PP theory and the anonymous reviewers and editor for their helpful comments.
}

\title{\TitleFull{}}
\author{\AuthorsFull{}}
\date{}
\begin{document}

\Head{}
\tableofcontents

\section{Introduction}
\label{sec_introduction}

\par
A \emph{point process} (PP) $\xi$ on a complete separable metric space is
\emph{$R$-dependent}, if events of $\xi$ based on Borel sets having mutual distance greater than or equal to $R$ are independent.
This work only deals with simple PPs.
Natural examples of $R$-dependent PPs are as follows.
A Poisson PP, which is even \emph{$0$-dependent}.
Range $R/2$ dependent thinnings of Poisson PPs à la Matérn~\cite{Matern__SpatialVariation_StochasticModelsAndTheirApplicationToSomeProblemsInForestSurveysAndOtherSamplingInvestigations__MFSS_1960,Stoyan_Stoyan__OnOneOfMaternsHardcorePointProcessModels__MN_1985}.
Poisson cluster PPs~\cite{Baddeley_VanLieshout_Moller__MarkovPropertiesOfClusterProcesses__AAP_1996,Daley_VereJones__AnIntroductionToTheTheoryOfPointProcesses_I__Springer_2003} with an offspring distribution supported on a sphere of radius $R/2$ around a cluster centre point.
Local constructions based on a Poisson PP, such as taking the centres of circumscribed circles of radius less than $R/2$ of triangles formed by triples of points from the Poisson PP.
Determinantal and permanental PPs~\cite{Borodin__DeterminantalPointProcesses__OUP_2011,Eisenbaum__StochasticOrderForAlphaPermanentalPointProcesses__SPA_2012,Soshnikov__DeterminantalRandomPointFields__UMN_2000} with a kernel of bounded range $R$ are also $R$-dependent.

\par
If the space is discrete, then a simple PP is a \emph{Bernoulli random field} (short BRF), an at most countable collection of $\EnumSet{0,1}$-valued random variables indexed by the space.
The study of $R$-dependent BRFs has a long history in the theory of discrete stochastic processes~\cite{Aaronson_Gilat_Keane__OnTheStructureOfOneDependentMarkovChains__JTP_1992,Aaronson_Gilat_Keane_DeValk__AnAlgebraicConstructionOfAClassOfOneDependentProcesses__AP_1989,Broman__OneDependentTrigonometricDeterminantalProcessesAreTwoBlockFactors__AP_2005,Burton_Goulet_Meester__OnOneDependentProcessesAndKBlockFactors__AP_1993,DeValk__TheMaximalAndMinimalTwoCorrelationOfAClassOfOneDependentZeroOneValuedProcesses__IJM_1988,DeValk__HilbertSpaceRepresentationsOfMDependentProcesses__AP_1993,Janson__RunsInMDependentSequences__AP_1984,Liggett_Schonmann_Stacey__DominationByProductMeasures__AoP_1997}.
Further uses of $R$-dependent BRFs are within the part of the probabilistic method in combinatorics building on the \emph{Lovász Local Lemma} (LLL)~\cite{Alon_Spencer__TheProbabilisticMethod_Ed3__Wiley_2008,Erdos_Lovasz__ProblemsAndResultsOn3ChromaticHypergraphsAndSomeRelatedQuestions__CMSJB_1975} and in graph colouring~\cite{Holroyd__OneDependentColoringByFinitaryFactors__Arxiv_2014,Jensen_Toft__GraphColoringProblems__Wiley_1995}.

\par
Let $(\Space,\Metric)$ be the complete separable metric space.
Without loss of generality, rescaling the metric reduces the discussion to $1$-dependence.
Let $\BorelBounded$ and $\BorelAll$ be the set of bounded and all Borel sets on $\Space$ respectively.
Let $\MeasuresBounded$ be the space of boundedly finite Borel measures.
The intensity measure of a PP is the expected number of points in a given Borel set $B$.
For $M\in\MeasuresBounded$, let $\OneDepPPM$ be the class of \emph{simple and $1$-dependent PP laws with intensity measure $M$}.
Where there is no danger of confusion, identify a PP and its law.
For $\xi\in\OneDepPPM$ and $B\in\BorelBounded$, the \emph{avoidance probability} is the probability of $\xi$ having no points in $B$.
The \emph{avoidance function} of $\xi$ maps $\BorelBounded$ to the avoidance probabilities of $\xi$.

\par
This paper studies uniform lower bounds on the avoidance function of PPs in $\OneDepPPM$.
First, Section~\ref{sec_dichotomy} extends a dichotomy by Shearer~\cite{Shearer__OnAProblemOfSpencer__Comb_1985} from BRFs to PPs.
For large intensity measure $M$, there exists a PP in $\OneDepPPM$ with an avoidance function vanishing on some bounded Borel set of positive $M$-measure (\emph{zero phase}).
For small intensity measure $M$, there is a uniform positive lower bound on the avoidance function of PPs in $\OneDepPPM$ (\emph{positive phase}).

\par
Second, for $M$ in the positive phase, Section~\ref{sec_spp} generalises a construction by Shearer~\cite{Shearer__OnAProblemOfSpencer__Comb_1985} of the unique PP in $\OneDepPPM$ with minimal avoidance function.
A set of points is \emph{$r$-hard-core}, if its points have mutual distance at least $r$.
A PP $\xi$ is $r$-hard-core, if its realisations are almost-surely so.
An $r$-hard-core and $R$-dependent PP must have $r\le R$.
$R$-dependence and $R$-hard-core together imply uniqueness of the PP law and existence only for small intensity measures.
If a PP with these properties exists, call it \emph{Shearer's PP}.
Shearer's PP has the minimal avoidance function in $\OneDepPPM$, because it avoids clusters of points and spreads its points all over space.
The existence of Shearer's PP describes the positive phase.

\par
Section~\ref{sec_hardspheres} recalls the \emph{hard-sphere model}, the unique Markov PP with range $R$ interaction and an $R$-hard-core.
The partition function of the hard-sphere model and the avoidance function of Shearer's PP have the same algebraic structure.
Shearer's PP exists for a given intensity measure $M$, if and only if the cluster expansion of the hard-sphere model converges uniformly and absolutely at negative fugacity $-M$, generalising the identification in the BRF case~\cite{Scott_Sokal__TheRepulsiveLatticeGasTheIndependentSetPolynomialAndTheLovaszLocalLemma__JSP_2005}.

\par
Third, Section~\ref{sec_lll} generalises the LLL~\cite{Erdos_Lovasz__ProblemsAndResultsOn3ChromaticHypergraphsAndSomeRelatedQuestions__CMSJB_1975} to the case of PPs.
Here, a LLL denotes a \emph{sufficient condition} on $\Space$ and $M$ for a \emph{uniform exponential lower bound} on the avoidance functions in $\OneDepPPM$ and to be in the positive phase.
The core idea is to derive global properties, i.e., being in the positive phase, from local properties.
On $\Reals^d$, this yields an explicit and uniform upper bound on the \emph{empty space functions} $F$~\cite[Section 15.1]{Daley_VereJones__AnIntroductionToTheTheoryOfPointProcesses_II__Springer_2008} and $J$ functions~\cite{Baddeley_VanLieshout__ANonparametricMeasureOfSpatialInteractionInPointPatterns__StatNL_1996} of isotropic and $1$-dependent PPs.
For the hard-sphere model, a LLL becomes a sufficient condition for the convergence of the cluster expansion and yields a lower bound on the radius of convergence.
This improves a classic lower bound via cluster expansion techniques by Ruelle~\cite{Ruelle__StatisticalMechanics_RigorousResults__Ben_1969} by a short inductive argument à la Dobrushin~\cite{Dobrushin__EstimatesOfSemiInvariantsForTheIsingModelAtLowTemperatures__AMST_1996} of less than two pages.

\par
Section~\ref{sec_stochastic_domination} explains the relevance of exponential lower bounds on the avoidance function for uniform stochastic domination by a Poisson PP.
Section~\ref{sec_onedep} discusses variations of $1$-dependence and a key inequality of the avoidance functions of $1$-dependent PPs.
The remaining sections contain proofs.

\section{Results}
\label{sec_results}

\subsection{Uniform bounds on the avoidance functions}
\label{sec_dichotomy}

\par
The first question is the existence of $1$-dependent PPs with a given intensity measure.
Let $\MeasuresEmpty := \DescSet{M\in\MeasuresBounded}{\OneDepPPM=\emptyset}$.
Proposition~\ref{prop_class_empty} shows that atoms of mass greater than one in the intensity measure are the only obstacle.
Thus, $\OneDepPPM\not=\emptyset$, if and only if $M\in\MeasuresBounded\setminus\MeasuresEmpty = \DescSet{M\in\MeasuresBounded}{\forall x\in\Space: M(\EnumSet{x})\le 1}$.

\par
This work is about lower bounds on the avoidance functions of $1$-dependent PPs.
The first result extends a dichotomy by Shearer~\cite{Shearer__OnAProblemOfSpencer__Comb_1985} from BRFs to PPs.

\begin{Thm}
\label{thm_dichotomy}
If $M\in\MeasuresBounded\setminus\MeasuresEmpty$, then it falls into one of two phases.
In the \emph{zero phase}, there is a $\xi\in\OneDepPPM$ with zero avoidance probability on some $B\in\BorelBounded$ of positive measure.
That is, $M(B)>0$ and $\Proba(\xi(B)=0) = 0$.
In the \emph{positive phase}, there is a unique $\DistinguishedMinimalPP\in\OneDepPPM$ minimizing the (conditional) avoidance probabilities uniformly in space and the class.
That is, for every $\xi\in\OneDepPPM$ and all $A,B\in\BorelBounded$, one has
\begin{equation*}
 \Proba(\xi(B)=0|\xi(A)=0)\ge\Proba(\DistinguishedMinimalPP(B)=0|\DistinguishedMinimalPP(A)=0)>0
 \,.
\end{equation*}
\end{Thm}

Theorem~\ref{thm_dichotomy} follows from Corollary~\ref{cor_strict_implies_positive} and Theorem~\ref{thm_nonstrict_implies_zero}.
Section~\ref{sec_spp} describes the distinguished PP $\DistinguishedMinimalPP$ in the positive phase.
Let $\MeasuresPhaseZero$ and $\MeasuresPhasePositive$ be the subsets of $\MeasuresBounded\setminus\MeasuresEmpty$ being in the zero phase and positive phase respectively.
See also Figure~\ref{fig_intensityCone}.

\subsection{The generating function}
\label{sec_gf}

\par
A \emph{configuration} is a countable collection of points in $\Space$.
A configuration $C$ is \emph{$1$-hard-core}, if its points have mutual distance at least one, i.e., for all $\EnumSet{x,y}\subseteq C: \Distance{x}{y}\ge 1$.
In the classic case of the metric space $(V,2d)$ derived from a graph $G:=(V,E)$ with geodesic metric $d$, the $1$-hard-core configurations are the graph-theoretic \emph{independent} sets of $G$.
For $n\in\NonNegInts$, let $\IsHCTuple{n}$ be the indicator function of $1$-hard-core tuples in $\Space^n$.

\begin{Def}
\label{def_gf}
The \emph{generating function $\GfSymbol$ of weighted $1$-hard-core configurations} is\\
$\BorelBounded\times\MeasuresBounded\to\Reals$ and
\begin{equation*}
 (B,M)\mapsto
 \sum_{n=0}^\infty \frac{(-1)^n}{n!}
 \int_{B^n} \IsHCTuple{n}(x_1,\dotsc,x_n) \prod_{i=1}^n M(\D{x_i})
 \,.
\end{equation*}
\end{Def}
Proposition~\ref{prop_gf_basic_properties} shows that $\GfSymbol$ is well-defined.
If $M$ is diffuse, then $\GfG{B}{M}$ is the expectation of a functional of a Poisson PP of intensity $M$.
Such a representation fails for intensity measures containing atoms.
The alternating sign in $\GfG{B}{M}$ is a convenience to avoid using a negative measure as argument to $\GfSymbol$ throughout most of this work.
For unambiguous choices of $M$, $\GfM{B}$ abbreviates $\GfG{B}{M}$.
While $\GfM{\emptyset}=1$ always holds, a central topic is which $M$ admit $\GfG{B}{M}\ge 0$ uniformly in $B$.
For all $A,B\in\BorelBounded$ with $\GfM{B}>0$, a key quantity is $\GfMR{A}{B}:=\GfM{A\cup B}/\GfM{B}$.
Section~\ref{sec_gf_properties} discusses the properties of $\GfSymbol$ and $\GfRSymbol$ in detail.

\subsection{Shearer's point process}
\label{sec_spp}

\par
This section is about $1$-dependent PPs with an $1$-hard-core.
For a given intensity measure, there is at most one such PP.
If it exists, call it \emph{Shearer's PP}, as it generalises the BRF construction of Shearer~\cite{Shearer__OnAProblemOfSpencer__Comb_1985}.

\begin{Thm}
\label{thm_spp}
If an $1$-hard-core $\ShearersPPM\in\OneDepPPM$ exists, then its law is unique.
Its avoidance function is $\GfSymbol$, i.e., for each $B\in\BorelBounded$, $\Proba(\ShearersPPM(B)=0)=\GfM{B}$.
Such a PP exists for all intensity measures in
\ArxivOnly{\\}{ }$\MeasuresShearerExists:=\DescSet{M\in\MeasuresBounded}{\forall B\in\BorelBounded: \GfG{B}{M}\ge 0}$.
\end{Thm}

\par
The proofs of this section's statements are in Section~\ref{sec_spp_proofs}.
If Shearer's PP exists, then it minimizes the (conditional) avoidance probabilities within the class of $1$-dependent PPs with the same intensity measure.

\begin{Thm}
\label{thm_minimality}
If $M\in\MeasuresShearerExists$, then, for all $\xi\in\OneDepPPM$ and $A, B\in\BorelBounded$ with $\GfM{B}>0$,
\begin{equation*}
 \Proba(\xi(A)=0|\xi(B)=0)\ge\GfMR{A}{B}\ge 0
 \,.
\end{equation*}
\end{Thm}

\par
If Shearer's PP has a positive avoidance function, then it is the unique PP $\DistinguishedMinimalPP$ from the positive phase in Theorem~\ref{thm_dichotomy}.
\ArxivOnly{\\}{ }
Let $\MeasuresShearerStrict:= \DescSet{ M\in\MeasuresBounded}{\forall B\in\BorelBounded: \GfG{B}{M}> 0 }$.
\begin{Cor}
\label{cor_strict_implies_positive}
If $M\in\MeasuresShearerStrict$, then $M\in\MeasuresPhasePositive$ and Shearer's PP $\ShearersPPM$ has the minimal (conditional) avoidance function in $\OneDepPPM$.
In short, $\MeasuresShearerStrict\subseteq\MeasuresPhasePositive$.
\end{Cor}
Corollary~\ref{cor_strict_implies_positive} is a direct consequence of Theorem~\ref{thm_minimality}.
On the other hand, if Shearer's PP does not have a positive avoidance function on $\BorelBounded$ or does not exist, then there is a PP with zero avoidance probability on some bounded Borel set.
This puts the intensity measure $M$ in the zero phase.
\begin{Thm}
\label{thm_nonstrict_implies_zero}
If $M\not\in\MeasuresShearerStrict$ and $M\not\in\MeasuresEmpty$, then $M\in\MeasuresPhaseZero$.
\end{Thm}
Theorem~\ref{thm_nonstrict_implies_zero} and Corollary~\ref{cor_strict_implies_positive} together imply that $\MeasuresShearerStrict=\MeasuresPhasePositive$.

\par
Independent thinning of Shearer's PP decreases the intensity measure and preserves the two characterising properties of $1$-dependence and $1$-hard-core.
A bigger intensity measure decreases the avoidance function or inhibits the existence of $1$-dependent PPs.
\begin{Thm}
\label{thm_monotonicity}
The sets $\MeasuresShearerExists$ and $\MeasuresShearerStrict=\MeasuresPhasePositive$ are \emph{down-sets}, i.e., closed under decreasing the measure.
The set $\MeasuresEmpty$ is an \emph{up-set}, i.e., closed under increasing the measure.
The set $\MeasuresPhaseZero$ is an up-set, as long as atoms do not increase beyond mass one.
\end{Thm}

\par
In general, Shearer's PP differs from all other $R$-independent models in the introduction.
Probabilistic constructions of Shearer's PP are known only in special cases.
Details are in Section~\ref{sec_spp_different}.

\subsection{The hard-sphere model}
\label{sec_hardspheres}

\par
Another simple $1$-hard-core PP related to the function $\GfSymbol$ is the \emph{hard-sphere model}.
It is a Markov PP with the most repulsive range $1$ interaction~\cite[Section 1.2.2]{Ruelle__StatisticalMechanics_RigorousResults__Ben_1969}.
A common visualisation of the hard-sphere model is as a collection of non-overlapping open spheres with radius $1/2$ representing the hard cores of atoms.

\par
The hard-sphere model $\HardSphereM{B}$ in a \emph{finite volume} $B\in\BorelBounded$ with \emph{fugacity} $M\in\MeasuresBounded$ and empty boundary conditions has Janossy intensity
\begin{equation*}
 \Proba(\HardSphereM{B} = \D{(x_1,\dotsc,x_n)})
 = \frac{\IsHCTuple{n}(x_1,\dotsc,x_n)}{\GfG{B}{-M}} \prod_{i=1}^n M(\D{x_i})
 \,.
\end{equation*}
The normalising factor $\GfG{B}{-M}$ is the \emph{partition function} of the hard-sphere model.
The argument $-M$ cancels the alternating sign in the definition of $\GfSymbol$.
For diffuse $M$, $\HardSphereM{B}$ equals a Poisson($M$) PP on $B$ conditioned to be $1$-hard-core.

\par
The analysis of the hard-sphere model centres on the partition function and derived quantities, in particular ratios (\emph{reduced correlations}) and its logarithm (\emph{free energy}).
Lower bounds on $\GfSymbol$ and $\GfRSymbol$ and their logarithms at negative fugacity play a key role in the low fugacity case (the \emph{high temperature} case) and establish uniqueness of the Gibbs measure~\cite{Fernandez_Procacci__ClusterExpansionForAbstractPolymerModels_NewBoundsFromAnOldApproach__CMP_2007,Scott_Sokal__TheRepulsiveLatticeGasTheIndependentSetPolynomialAndTheLovaszLocalLemma__JSP_2005}.
A well-known tool is the \emph{cluster expansion}, a series expansion of $\log\GfSymbol$~\cite{MiracleSole__OnTheTheoryOfClusterExpansions__MPRF_2010}.
It fails first at negative fugacities, that is along the boundary of $\MeasuresShearerStrict$.
Details are in Section~\ref{sec_ce}.

\subsection{Sufficient conditions for exponential bounds}
\label{sec_lll}

\par
This section contains several LLLs in Theorems~\ref{thm_lll_rough} and~\ref{thm_lll_smooth} and Corollary~\ref{cor_lll_lebesgue}.
Theorem~\ref{thm_lll_generic} discusses the relation between $\MeasuresShearerStrict$ and $\MeasuresShearerExists$.
The proofs are in Section~\ref{sec_lll_proofs}.

\par
Let $\BorelUnitDiam$ be the set of Borel sets of diameter less than one.
Let $\sqcup$ denote disjoint union.
The \emph{unit partition number} of a Borel set $B\in\BorelAll$ is
\begin{equation*}
 \UPartNumOf{B}:=
 \inf\BigDescSet{k\in\Naturals}%
 {\exists A_1,\dotsc,A_k\in\BorelUnitDiam: B=\bigsqcup_{i=1}^k A_i }
 \,.
\end{equation*}
Let $\UnitSphere{x}$ be the \emph{open unit sphere} around $x$.
The nice case are spaces where the unit-scale structure has an \emph{uniform exponential growth bound} on spheres of radius $r$.
Spaces such as $\Reals^d$, the hyperbolic plane or graphs with uniformly bounded degree fulfil this, but infinite dimensional metric spaces such as $l_2$ do not.
This is equivalent to the bound
\begin{equation}
\label{eq_upartnum_finite}
 \UPartSup
 :=
 \sup\DescSet{\UPartNumOf{\UnitSphere{x}}}{ x\in\Space}
 <\infty
 \,.
\end{equation}

\par
The first sufficient condition generalises the \emph{symmetric LLL}~\cite{Erdos_Lovasz__ProblemsAndResultsOn3ChromaticHypergraphsAndSomeRelatedQuestions__CMSJB_1975} and \emph{Dobrushin's condition}~\cite{Dobrushin__EstimatesOfSemiInvariantsForTheIsingModelAtLowTemperatures__AMST_1996} respectively.

\begin{Thm}
\label{thm_lll_rough}
\begin{subequations}
\label{eq_lll_rough}
Let $M\in\MeasuresBounded$.
If~\eqref{eq_upartnum_finite} holds and, for each $A\in\BorelUnitDiam$,
\begin{equation}
\label{eq_lll_rough_measure}
 M(A)\le\frac{(\UPartSup+1)^{\UPartSup+1}}{(\UPartSup+2)^{\UPartSup+2}}
 \,,
\end{equation}
then $M\in\MeasuresShearerStrict$ and, for all $A,B\in\BorelBounded$,
\begin{equation}
\label{eq_lll_rough_bound}
 \GfMR{A}{B}
 \ge\left(\frac{\UPartSup+1}{\UPartSup+2}\right)^{\UPartNumOf{A\setminus B}}
 >0
 \,.
\end{equation}
\end{subequations}
\end{Thm}

\par
Under condition~\eqref{eq_upartnum_finite}, a slight loss of precision sharpens the positive lower bound in the positive phase to exponential lower bounds. Proposition~\ref{thm_lll_generic} shows that $\MeasuresShearerExists$ is a subset of the closure of $\MeasuresShearerStrict$.
In general, $\MeasuresShearerExists$ is not the closure of $\MeasuresShearerStrict$~\cite[Chapter 8]{Scott_Sokal__TheRepulsiveLatticeGasTheIndependentSetPolynomialAndTheLovaszLocalLemma__JSP_2005}.
It is so on a finite graph, but already fails on a connected infinite graph.

\begin{Thm}
\label{thm_lll_generic}
Let $\alpha>0$.
If $(1+\alpha)M\in\MeasuresShearerExists$ and~\eqref{eq_upartnum_finite} holds, then $M\in\MeasuresShearerStrict$ and, for all $A,B\in\BorelBounded$,
$\GfMR{A}{B}
 \ge\left(\frac{\alpha}{1+\alpha}\right)^{\UPartNumOf{A\setminus B}}
 >0
$.
\end{Thm}

\par
The second sufficient condition generalises the \emph{asymmetric LLL}~\cite{Erdos_Lovasz__ProblemsAndResultsOn3ChromaticHypergraphsAndSomeRelatedQuestions__CMSJB_1975}.
It resembles a \emph{continuous version of the Koteck\'{y}-Preiss condition}~\cite{Kotecky_Preiss__ClusterExpansionForAbstractPolymerModels__CMP_1986}.

\begin{Thm}
\label{thm_lll_smooth}
Let $M,N\in\MeasuresBounded$ with $N$ being absolutely continuous with respect to $M$.
If, for each $A\in\BorelUnitDiam$,
\begin{subequations}
\label{eq_lll_smooth}
\begin{equation}
\label{eq_lll_smooth_condition}
 \int_A \exp(N(\UnitSphere{x}\setminus A)) M(\D{x}) \le 1 - \exp(-N(A))
 \,,
\end{equation}
then $M\in\MeasuresShearerStrict$ and, for all $A,B\in\BorelBounded$,
\begin{equation}
\label{eq_lll_smooth_bound}
 \GfMR{A}{B}\ge\exp(-N(A\setminus B))>0
 \,.
\end{equation}
\end{subequations}
\end{Thm}

\par
A stronger and more practical version of condition~\eqref{eq_lll_smooth_condition} is
\begin{equation}
\label{eq_lll_smooth_condition_stronger}
 \forall A\in\BorelUnitDiam, M\text{-a.e.
}x\in A:
 \quad
 M(A)\exp(N(\UnitSphere{x}\setminus A))\le 1-\exp(-N(A))
 \,.
\end{equation}
Specialising~\eqref{eq_lll_smooth_condition_stronger} to the space $\Reals^d$ with the Lebesgue measure yields an explicit bound.
\begin{Cor}
\label{cor_lll_lebesgue}
Consider $\Space=\Reals^d$ with the Lebesgue measure $\Lebesgue$.
Let $V$ be the volume of the $d$-dimensional unit sphere.
For $\lambda\le 1/(e V)$, let $\alpha$ be the unique solution of $\lambda=\alpha\exp(-\alpha V)$ in $[0,1/(eV)]$.
This implies that $\lambda\Lebesgue\in\MeasuresShearerStrict$ and, for all $A,B\in\BorelBounded$, $\GfMR{A}{B}\ge \exp(-\alpha\Lebesgue(A\setminus B))>0$.
\end{Cor}

\par
In the context of the hard-sphere model, Corollary~\ref{cor_lll_lebesgue} yields a uniform upper bound on finite volume free energies.
That is, $\sup\BigDescSet{-\frac{\log\GfG{B}{-\lambda\Lebesgue}}{\Lebesgue(B)}}{B\in\BorelBounded}\le \alpha$.
Together with taking the limit of $-\frac{\log\GfG{B}{-\lambda\Lebesgue}}{\Lebesgue(B)}$ in the van Hove sense~\cite[Definition 2.1.1]{Ruelle__StatisticalMechanics_RigorousResults__Ben_1969}, this implies the existence and complete analyticity of the infinite volume free energy for fugacities less than $1/(e V)$.
Thus, Corollary~\ref{cor_lll_lebesgue} gives an alternative proof, avoiding cluster expansion, of a classic result for uniqueness of the Gibbs measure of the hard-sphere model on $\Reals^d$ for small fugacities~\cite[(5.2) in Section 4.5]{Ruelle__StatisticalMechanics_RigorousResults__Ben_1969}.

\par
The LLLs and derived bounds here are not optimal.
For example, the optimal bound on $\Reals$ in the context of Corollary~\ref{cor_lll_lebesgue} is $1/e$ instead of $1/(2e)$~\cite{HoferTemmel__ShearersPointProcessAndTheHardSphereModelInOneDimension}.
There are two reasons to not pursue further improvements of the LLLs here.
First, improvements in the main cases of interest are already present in the literature on the hard-core and hard-sphere models~\cite{Fernandez_Procacci__ClusterExpansionForAbstractPolymerModels_NewBoundsFromAnOldApproach__CMP_2007,Fernandez_Procacci_Scoppola__TheAnalyticityRegionOfTheHardSphereGas_ImprovedBounds__JSP_2007}.
Second, in the context of Corollary~\ref{cor_lll_lebesgue}, the best bounds differ by at most a factor of $e$~\cite[(5.17) in Section 4.5]{Ruelle__StatisticalMechanics_RigorousResults__Ben_1969}.

\begin{figure}
\begin{center}
\FigureIntensityCone{1.3}
\end{center}
\caption{
The cone of boundedly finite Borel measures $\MeasuresBounded$, with diffuse measures on the left axis, atomic measures on the right axis and the zero measure in the apex.
The positive and zero phases ($\MeasuresPhasePositive$ and $\MeasuresPhaseZero$) and $\MeasuresEmpty$ partition the cone.
The sets $\MeasuresPhasePositive$, $\MeasuresShearerExists$ and $\MeasuresShearerStrict$ are down-sets, whereas $\MeasuresEmpty$ and $\MeasuresPhaseZero\sqcup\MeasuresEmpty$ are up-sets.
The dotted line represents the atypical closure properties of $\MeasuresShearerStrict$ and contains $\MeasuresShearerExists\setminus\MeasuresShearerStrict$.
The LLLs carve out parts of the positive phase and guarantee exponential lower bounds.
}
\label{fig_intensityCone}
\end{figure}

\subsection{Stochastic domination and order}
\label{sec_stochastic_domination}

\par
A PP law $\varphi$ \emph{stochastically dominates} a PP law $\xi$, if there is a coupling of them such that $\varphi$ contains almost surely all of $\xi$'s points.
Stochastic domination implies that $\varphi$'s avoidance function is smaller than $\xi$'s avoidance function.
In the context of $1$-dependent BRFs on locally finite graphs, the existence of Shearer's PP is equivalent to uniform stochastic domination by a Bernoulli product field ~\cite{Liggett_Schonmann_Stacey__DominationByProductMeasures__AoP_1997,Temmel__ShearersMeasureAndStochasticDominationOfProductMeasures__JTP_2014}.
Uniform exponential lower bounds from the LLLs are a first step towards extending the stochastic domination result to $1$-dependent PPs.

\par
Stochastic domination of Shearer's PP by a Poisson PP would permit a probabilistic construction and simulation by thinning the dominating Poisson PP.
The intrinsic coupling of Shearer's PP by independent thinning implies that stochastic domination for the largest intensities in $\MeasuresShearerExists$ suffices~\cite{Rolski_Szekli__StochasticOrderingAndThinningOfPointProcesses__SPA_1991}.

\par
Using the terminology from~\cite{Blaszczyszyn_Yogeshwaran__OnComparisonOfClusteringPropertiesOfPointProcesses__AAP_2014}, $\ShearersPPM$ is \emph{weakly sub Poisson}.
This means that its moments and avoidance function are both smaller than the ones of a Poisson($M$) PP.
The first follows from Proposition~\ref{prop_spp_characterisation_uniqueness} and~\eqref{eq_int_bound}.
The second follows by applying~\cite[Prop 3.1]{Blaszczyszyn_Yogeshwaran__OnComparisonOfClusteringPropertiesOfPointProcesses__AAP_2014} to the implication of Proposition~\ref{prop_monotonicity}, that, for disjoint $A,B\in\BorelBounded$,
\begin{equation*}
 \Proba(\ShearersPPM(A\sqcup B)=0)
 = \GfG{A\sqcup B}{M}
 \le \Proba(\ShearersPPM(A)=0)\Proba(\ShearersPPM(B)=0)
 \,.
\end{equation*}
Being weakly-sub Poisson yields concentration inequalities~\cite[Sec 3.3]{Blaszczyszyn_Yogeshwaran__ClusteringComparisonOfPointProcessesWithApplicationsToRandomGeometricModels__Springer_2015}.

\section{About one-dependent PPs}
\label{sec_onedep}

\par
This section discusses different notions of $1$-dependence in~\eqref{eq_onedep}.
Proposition~\ref{prop_fi} presents a key inequality of avoidance functions of $1$-dependent PPs.
Proposition~\ref{prop_strongOneMulti} characterises the avoidance functions of $1$-dependent PPs.
Proposition~\ref{prop_class_empty} investigates the existence of simple $1$-dependent PPs.

\par
\begin{subequations}
\label{eq_onedep}

For $A\in\BorelAll$ and a PP $\xi$, let $A\xi$ be the restriction of $\xi$ to $A$.
Recall that the metric is $\Metric$.
A PP $\xi$ is \emph{strong $1$-dependent}, if, for all $A,B\in\BorelAll$,
\begin{equation}
\label{eq_onedep_strong}
 \Distance{A}{B}:=\inf\DescSet{\Distance{x}{y}}{ x\in A, y\in B}\ge 1
 \Then
 A\xi\text{ is independent of }B\xi
 \,.
\end{equation}
All the examples in the introduction, the PP in the zero phase of Theorem~\ref{thm_dichotomy}, the PP counterexample in the proof of Theorem~\ref{thm_nonstrict_implies_zero}, Shearer's PP and the PP in Proposition~\ref{prop_class_empty} are strong $1$-dependent.

\begin{Prop}
\label{prop_strongOneMulti}
A PP is strong $1$-dependent, if and only if its avoidance function $Q$ is \emph{$1$-multiplicative}.
That is, for all $A,B\in\BorelAll$, $\Distance{A}{B}\ge 1 \Then Q(A\cup{}B)=Q(A)Q(B)$.
\end{Prop}

\begin{proof}
The necessity is evident, and the sufficiency follows from the fact that the avoidance function determines the law of a simple PP~\cite[9.2.XIII]{Daley_VereJones__AnIntroductionToTheTheoryOfPointProcesses_II__Springer_2008}.
\end{proof}

A PP $\xi$ is \emph{weak $1$-dependent}, if, for $M$-a.e.
$x$ and every $B\in\BorelAll$,
\begin{equation}
\label{eq_onedep_weak}
 \Distance{x}{B}:=\Distance{\EnumSet{x}}{B}\ge 1
 \Then
 \Proba_x(\xi(B)=0)=\Proba(\xi(B)=0)
 \,.
\end{equation}
In other words, weak $1$-dependence reduces Palm probabilities for events at distance more than one away from the base point to normal probabilities.
\end{subequations}

\begin{Prop}
\label{prop_strong_is_weak}
A strong $1$-dependent PP is weak $1$-dependent.
\end{Prop}

Examples of weak, but not strong, $1$-dependent PPs are mixtures of random shifts of strong $1$-dependent BRFs in~\cite[Section 5]{Liggett_Schonmann_Stacey__DominationByProductMeasures__AoP_1997}.

\begin{proof}
Let $\xi\in\OneDepPPM$.
The \emph{Campbell measure} $C$~\cite[(13.1.1a)]{Daley_VereJones__AnIntroductionToTheTheoryOfPointProcesses_II__Springer_2008}
on $(B,E)\in\BorelAll\times\sigma(\xi)$ is $C(B\times E):=\sum_{n=1}^\infty n\Proba(\xi(B)=n,\xi\in E)$.
If $A,B\in\BorelBounded$ with $\Distance{A}{B}\ge 1$ and $E\in\sigma(A\xi)$ with embedding $E'$ into $\sigma(\xi)$, then strong $1$-dependence~\eqref{eq_onedep_strong} allows to factorise
$
 C(B\times E')
 =\sum_{n=1}^\infty n\Proba(\xi(B)=n,A\xi\in E)
 =\Expect(\xi(B))\Proba(A\xi\in E)
$.
Hence, $M$-a.e., the Palm density on $\sigma(A\xi)$ simplifies to
$
 \Proba_x(\xi\in E)
 =\frac{dC(.\times E)}{d\Expect(\xi(\EnumSet{.}))}
 =\Proba(\xi\in E)
$.
As the $\sigma$-algebras of the form $\sigma(A\xi)$, for $A\in\BorelBounded$ with $A\cap\UnitSphere{x}=\emptyset$, generate $\sigma((X\setminus\UnitSphere{x})\xi)$, weak $1$-dependence follows.
\end{proof}

The inequality in Proposition~\ref{prop_fi} lies at the core of the dichotomy in Theorem~\ref{thm_dichotomy} and the minimality of the avoidance function of Shearer's PP in Theorem~\ref{thm_minimality}.
Proposition~\ref{prop_fe} shows that the $1$-hard-core of Shearer's PP makes it the only $1$-dependent PP to fulfil the inequality uniformly as an equality.
The inequality motivates the relaxation from strong to weak $1$-dependence.
Thus, $\OneDepPPM$ may be extended to be the class of weak $1$-dependent simple PPs with intensity measure $M$.

\begin{Prop}
\label{prop_fi}
For all $\xi\in\OneDepPPM$, $A\in\BorelUnitDiam$ and $B\in\BorelBounded$ with $\Proba(\xi(B)=0)>0$,
\begin{equation*}
 \frac{\Proba(\xi(A\cup B)=0)}{\Proba(\xi(B)=0)}
 \ge 1 - \int_{A\setminus B}
  \frac{\Proba(\xi(B\setminus\UnitSphere{x})=0)}{\Proba(\xi(B)=0)} M(\D{x})
 \,.
\end{equation*}
\end{Prop}

\begin{proof}
With $C$ the Campbell measure and the event $E:=\DescSet{\omega}{\xi(B)=0}$,
\begin{align*}
 \Proba(\xi(A\setminus B)\ge 1, \xi(B)=0)
 \le{}&\sum_{n=1}^\infty n\Proba(\xi(A\setminus B)=n, \xi\in E)
 \\={}& C((A\setminus B)\times E)
 \\={}&\int_{A\setminus B} \Proba_x(\xi(B)=0) M(\D{x})
 \\\le{}&\int_{A\setminus B} \Proba_x(\xi(B\setminus\UnitSphere{x})=0) M(\D{x})
 \\\stackrel{\eqref{eq_onedep_weak}}{=}{}&\int_{A\setminus B} \Proba(\xi(B\setminus\UnitSphere{x})=0) M(\D{x})
\end{align*}
and
\begin{align*}
 \Proba(\xi(A\cup B)=0)
 ={}& \Proba(\xi(B)=0) - \Proba(\xi(A\setminus B)\ge 1,\xi(B)=0)
 \\\ge{}& \Proba(\xi(B)=0)
  - \int_{A\setminus B} \Proba(\xi(B\setminus\UnitSphere{x})=0) M(\D{x})
 \,.\qedhere
\end{align*}
\end{proof}

\begin{Prop}
\label{prop_class_empty}
Let $M\in\MeasuresBounded$.
If $M$ has no atom of mass greater than one, then there exists a strong $0$-dependent PP with intensity measure $M$.
If $M$ has an atom of mass greater than one, then no simple PP with intensity measure $M$ exists.
\end{Prop}

\begin{proof}
For each $\xi\in\OneDepPPM$ and atom $x$ of $M$, $1\ge\Expect(\xi(\EnumSet{x}))= M(\EnumSet{x})=:m_x$.
Hence, an atom of mass greater than one contradicts simpleness of the PP.

\par
For the converse, let $M\in\MeasuresBounded$ without atoms greater than $1$.
Let $\Atoms$ and $\DiffuseDomain$ be the atomic and diffuse support domains of $M$ respectively.
Let $\Atoms_{=}$ and $\Atoms_{<}$ be the locations of atoms of mass equal to or less than one respectively.
Construct a measure $N$ with the same atomic and diffuse domains as follows.
On $\DiffuseDomain$, let $N\cap\DiffuseDomain:=M\cap\DiffuseDomain$.
On $\Atoms$, let $n_x:=N(\EnumSet{x}):=-\log(1-m_x)$, if $x\in\Atoms_{<}$, and $0$ else.
For $B\in\BorelBounded$, let $m_B:=\max\DescSet{m_x}{ x\in B\cap\Atoms_{<}}$.
As $M(B)<\infty$, so is $m_B<1$.
Also, $N(B\setminus\Atoms)=M(B\setminus\Atoms)<\infty$.
The inequality
\begin{equation}
\label{eq_log_inequality}
 \forall y\in[0,1[:
 \qquad
 -\log(1-y)
 =\sum_{n=1}^\infty \frac{y^n}{n}
 \le \sum_{n=1}^\infty y^n
 = \frac{y}{1-y}
\end{equation}
entails the bounded finiteness of $N$ on $\Atoms$ via
\begin{equation*}
 N(B\cap\Atoms)
 = -\sum_{x\in B\cap\Atoms_{<}} \log(1-m_x)
 \stackrel{\eqref{eq_log_inequality}}{\le}
   \sum_{x\in B\cap\Atoms_{<}} \frac{m_x}{1-m_x}
 \le \frac{M(B\cap\Atoms_{<})}{1-m_B}
 <\infty\,.
\end{equation*}

\par
The aim is to construct a strong $0$-dependent PP with intensity $M$.
Let $\varphi$ be the Poisson PP with intensity $N$.
It may not be simple because of atoms in $N$.
Let $\SupportPPOf{\varphi}$ be its simple support PP, collapsing multiple points of $\varphi$.
Let $\xi:=\SupportPPOf{\varphi}+\sum_{x\in\Atoms_{=}}\delta_x$.
The $0$-dependence of $\xi$ holds by construction and it remains to verify that $\xi$ has intensity $M$.
On the diffuse part of $M$, $\varphi$ is simple, whence the intensity of $\xi$ is $M$.
On $\Atoms_{=}$, $\varphi$ is almost-surely zero, but the atoms of mass one are present deterministically in $\xi$.
For all $x\in\Atoms_{<}$,
\begin{equation*}
 \Expect(\xi(\EnumSet{x}))
 = \Proba(\xi(\EnumSet{x})=1)
 = 1-\Proba(\varphi(\EnumSet{x})=0)
 = 1-e^{-n_x}
 = m_x
 \,.
\end{equation*}
Hence, for each $B\subseteq\BorelBounded$, the intensity measure is
\begin{equation*}
 \Expect(\xi(B\cap\Atoms_{<}))
 = \sum_{x\in B\cap\Atoms_{<}} m_x
 = M(B\cap\Atoms_{<})
 \,.\qedhere
\end{equation*}
\end{proof}

\section{Properties of the generating function}
\label{sec_gf_properties}

\par
Most properties of $\GfSymbol$ and $\GfRSymbol$ are trivial, if one knows that $\GfSymbol$ is the avoidance function of Shearer's PP.
The properties are needed to establish the existence of Shearer's PP first, though.
The equality in Proposition~\ref{prop_fe} (cf. the inequality in Proposition~\ref{prop_fi}) and monotonicity in Proposition~\ref{prop_monotonicity} are the most important ones.

\par
Let $\IntRange{n}:=\EnumSet{1,\dotsc,n}$, with $\IntRange{0}:=\emptyset$.
For a set $S$, let $S^n$ be the Cartesian product of $n$ copies of $S$, with $S^0:=\emptyset$.
Empty products evaluate to $1$ and empty sums to $0$.
For $B\in\BorelBounded$, $M\in\MeasuresBounded$ and $\lambda\in[0,\infty[$, let $\lambda M$ be the scaling of $M$ by the factor $\lambda$ and $M|_B$ the restriction of $M$ to $B$.

\subsection{Basic properties}
\label{sec_gf_basics}

\par
For $B\in\BorelBounded$ and $M\in\MeasuresBounded$, a basic bound is
\begin{equation}
\label{eq_int_bound}
 \int_{B^n} \IsHCTuple{n}(x_1,\dotsc,x_n)\prod_{i=1}^n M(\D{x_i})
 \le M(B)^n
 \,.
\end{equation}
Also, for each $B\in\BorelFinite$, $n\in\NonNegInts$ and $x_1,\dotsc,x_n\in B$,
\begin{equation}
\label{eq_hardcoreConfUPartBound}
 \IsHCTuple{n}(x_1,\dotsc,x_n) = 1 \Iff n\le\UPartNumOf{B}
 \,.
\end{equation}

\begin{Prop}
\label{prop_gf_basic_properties}
$\GfSymbol$ is well-defined and $1$-multiplicative as in Proposition~\ref{prop_strongOneMulti}.
\end{Prop}

\begin{proof}
For $k\in\NonNegInts$, the bound~\eqref{eq_int_bound} implies that
\begin{equation*}
 \BigModulus{
  \sum_{n=k}^\infty \frac{(-1)^n}{n!}
  \int_{B^n}\IsHCTuple{n}(x_1,\dotsc,x_n)\prod_{i=1}^n M(\D{x_i})
 }
 \le\sum_{n=k}^\infty \frac{M(B)^n}{n!}
 \,.
\end{equation*}
For $k=0$, this yields $\Modulus{\GfM{B}}\le\exp(M(B))$.
For $k\to\infty$, this shows the convergence of the series $\GfM{B}$.
\par
Let $A,B\in\BorelBounded$ with $\Distance{A}{B}\ge 1$.
The $1$-hard-core condition trivially holds for pairs in $A\times B$.
For $n,m\in\NonNegInts$, $x_1,\dotsc,x_n\in A$ and $y_1,\dotsc,y_m\in B$,
$\IsHCTuple{n+m}(x_1,\dotsc,x_n,y_1,\dotsc,y_m)
 = \IsHCTuple{n}(x_1,\dotsc,x_n)\IsHCTuple{m}(y_1,\dotsc,y_m)
$.
Hence, $1$-multiplicativity follows from
\begin{align*}
 \GfM{A\cup{}B}={}
 &\sum_{n=0}^\infty \frac{1}{n!}
  \int_{(A\cup{}B)^n} (-1)^n\IsHCTuple{n}(x_1,\dotsc,x_n)\prod_{i=1}^n M(\D{x_i})
 \\=
 &\sum_{n=0}^\infty \frac{1}{n!}
  \sum_{j=0}^n \binom{n}{j}
  \left(
  \int_{A^j} (-1)^j \IsHCTuple{j}(x_1,\dotsc,x_j) \prod_{i=1}^j M(\D{x_i})
  \right)
 \\&\qquad\qquad\qquad\times\left(
  \int_{B^{n-j}} (-1)^{n-j} \IsHCTuple{n-j}(x_1,\dotsc,x_{n-j})
   \prod_{i=1}^{n-j} M(\D{x_i})
  \right)
 \\={}
 &\left(
   \sum_{n=0}^\infty \frac{1}{n!}
   \int_{A^n} (-1)^n\IsHCTuple{n}(x_1,\dotsc,x_n) \prod_{i=1}^n M(\D{x_i})
  \right)
 \\&\times
  \left(
   \sum_{n=0}^\infty \frac{1}{n!}
   \int_{B^n} (-1)^n\IsHCTuple{n}(x_1,\dotsc,x_n) \prod_{i=1}^n M(\D{x_i})
  \right)
 \\={}&\GfM{A}\GfM{B}
 \,.\qedhere
\end{align*}
\par
\end{proof}

\begin{Prop}
\label{prop_fe}
For all $A\in\BorelUnitDiam$ and $B\in\BorelBounded$ with $\GfM{B}>0$,
\begin{equation}
\label{eq_fe}
 \GfMR{A}{B}
 = 1
 - \int_{A\setminus B}
    \GfMR{B\cap\UnitSphere{x}}{B\setminus\UnitSphere{x}}^{-1} M(\D{x})
 \,.
\end{equation}
\end{Prop}

\begin{proof}
By~\eqref{eq_hardcoreConfUPartBound}, at most one point of an $1$-hard-core configuration $C$ lies in $A$.
If $y\in A\cap C$, then the $1$-hard-core implies that $C\cap (B\setminus\UnitSphere{y})=\emptyset$.
This leads to
\begin{align*}
 \GfM{B}={}
 &\sum_{n=0}^\infty \frac{(-1)^n}{n!}
  \int_{B^n} \IsHCTuple{n}(x_1,\dotsc,x_n)\prod_{i=1}^n M(\D{x_i})
 \\={}
 &\sum_{n=0}^\infty \frac{(-1)^n}{n!}
  \left(
  \int_{(B\setminus A)^n} \IsHCTuple{n}(x_1,\dotsc,x_n)\prod_{i=1}^n M(\D{x_i})
  \right.
 \\&\qquad\qquad
  \left.
  +\,
  n\int_{A\times B^{n-1}} \IsHCTuple{n}(x_1,\dotsc,x_n)\prod_{i=1}^n M(\D{x_i})
  \right)
\end{align*}
A point in $A$ excludes the possibility of other points in $A$.
Thus,
\begin{align*}
 \GfM{B}={}
 &\sum_{n=0}^\infty \frac{(-1)^n}{n!}
  \int_{(B\setminus A)^n} \IsHCTuple{n}(x_1,\dotsc,x_n)\prod_{i=1}^n M(\D{x_i})
 \\
 &-\int_A \sum_{n=1}^\infty \frac{(-1)^{n-1}}{(n-1)!}
  \int_{B^{n-1}} \IsHCTuple{n}(x_1,\dotsc,x_{n-1},y)\prod_{i=1}^{n-1} M(\D{x_i})
  M(\D{y})
 \\={}
 &\GfM{B\setminus A}
  -\int_A
   \sum_{n=0}^\infty \frac{(-1)^n}{n!}
   \int_{(B\setminus\UnitSphere{y})^n} \IsHCTuple{n}(x_1,\dotsc,x_n)
   \prod_{i=1}^n M(\D{x_i})
  M(\D{y})
 \\={}
 &\GfM{B\setminus A} - \int_A \GfM{B\setminus\UnitSphere{y}} M(\D{y})
 \,.\qedhere
\end{align*}
\end{proof}

\subsection{Cluster expansion and monotonicity}
\label{sec_ce}

Cluster expansion is a series expansion of the logarithm of a generating series.
It is a classic technique from statistical mechanics~\cite[Section 4.4]{Ruelle__StatisticalMechanics_RigorousResults__Ben_1969}.

\begin{Prop}
\label{prop_ce}
Let $A,B\in\BorelFinite$ and $M\in\MeasuresBounded$ with $M|_{\Space\setminus(A\cup B)}=0$, i.e., it is concentrated on $A\cup B$.
The statement $\forall N\le M: N\in\MeasuresShearerStrict$ holds, if and only if
\begin{equation*}
 \log\GfR{A}{B}{M} = - \sum_{n=1}^\infty \frac{1}{n!}
  \int_{(A\cup B)^n\setminus B^n} \Penrose{x_1,\dotsc,x_n} \prod_{i=1}^n M(\D{x_i})
\end{equation*}
is a convergent series, with $\Penrose{x_1,\dotsc,x_n}\in\NonNegInts$ well-defined.
\end{Prop}

\begin{proof}
A cluster expansion of the partition function of a hard-sphere gas with radius one at negative fugacity~\cite{MiracleSole__OnTheTheoryOfClusterExpansions__MPRF_2010} with an application of Penrose's identity~\cite{Penrose__ConvergenceOfFugacityExpansionsForClassicalSystems__SMFA_1967} implies that the coefficients $\Penrose{x_1,\dotsc,x_n}$ have the desired properties.
\end{proof}

\ArxivOnly{
Some background about the coefficients $\Penrose{x_1,\dotsc,x_n}$.
Create a graph $G$ with vertices $\IntRange{n}$ and edges $(i,j)\in G$, if $\Distance{x_i}{x_j}<1$.
If $G$ is connected, then regard the term
\begin{equation*}
 \sum_{H\text{ spans }G} (-1)^{\Cardinality{E(H)}}\,.
\end{equation*}
Penrose shows that this counts the cardinality of a subset of the spanning trees of $G$, modulo a sign.
Let $\Penrose{x_1,\dotsc,x_n}$ be this count, if $G$ is connected, and $0$ otherwise.
This work only uses the non-negativity and finiteness of $\Penrose{x_1,\dotsc,x_n}$.}

\begin{Prop}
\label{prop_monotonicity}
As long as they are positive, the functions $\GfSymbol$ and $\GfRSymbol$ are monotone decreasing in both space and measure.
\end{Prop}

\begin{proof}
As $-\log\GfR{A}{B}{M}$ is a sum over integrals over non-negative integrands, it is monotone increasing in both the integration domains and the measure.
Because $\GfM{\emptyset}=1$, the same holds for the cluster expansion of $\GfSymbol$.
\end{proof}

Monotonicity implies the following telescoping identity.
For all $A,B\in\BorelBounded$ with $\GfM{B}>0$ and every partition $\EnumSet{A_i}_{i=1}^{n}$ of $A\setminus B$ by elements of $\BorelBounded$,
\begin{equation}
\label{eq_telescope}
 \GfMR{A}{B}
 = \prod_{i=1}^{n}\GfMR{A_i}{B\sqcup\bigsqcup_{j=1}^{i-1} A_j}
 \,.
\end{equation}

\subsection{Continuity properties}
\label{sec_continuity}

This section investigates continuity properties of $\GfSymbol$ in both space and measure.

\begin{Prop}
\label{prop_continuity_functional}
For $B\in\BorelBounded$ and $M\in\MeasuresBounded$, the function $\FunctionalM{B}:[0,\infty[\to\Reals, \lambda\mapsto\GfG{B}{\lambda M}$ is continuous.
\end{Prop}

\begin{proof}
The scaling of $M$ by $\lambda$ implies that continuity of $\FunctionalM{B}$ at $1$ suffices.
This follows from the bound, for each $\varepsilon$ with $\Modulus{\varepsilon}<1$,
\begin{align*}
 &\Modulus{\GfG{B}{(1+\varepsilon)M}-\GfG{B}{M}}
 \\={}
 &\BigModulus{
  \sum_{n=0}^\infty \frac{(-1)^n}{n!}
   \int_{B^n} \IsHCTuple{n}(x_1,\dotsc,x_n) (1-(1-\varepsilon)^n)
   \prod_{i=1}^n M(\D{x_i})
  }
 \\\le{}
 &\sum_{n=1}^\infty \frac{1}{n!}
  \sum_{j=1}^n \binom{n}{j} \Modulus{\varepsilon}^j
  \int_{B^n} \prod_{i=1}^n M(\D{x_i})
 \\\stackrel{~\eqref{eq_int_bound}}{\le}{}
 &\sum_{n=0}^\infty \frac{1}{n!} 2^n \Modulus{\varepsilon} M(B)^n
 ={}
 \Modulus{\varepsilon} \exp(2M(B))
 \,.\qedhere
\end{align*}
\end{proof}

\begin{Prop}
\label{prop_root}
For $B\in\BorelBounded$, let $\RootM{B}:=\inf\DescSet{\lambda}{ \FunctionalM{B}(\lambda)< 0}$.
If $B\supseteq A\in\BorelBounded$, then $\RootM{B}\le\RootM{A}$.
If $M(B)>0$, then $\RootM{B}=\min\DescSet{\lambda}{\FunctionalM{B}(\lambda)=0}\in\,]0,\infty[$ and $\RootM{B}M|_B\in\MeasuresShearerExists$.
\end{Prop}

\begin{proof}
\par
First prove the contra-variance of $\RootM{B}$.
Let $\Lambda<\lambda_B$.
Proposition~\ref{prop_ce} implies that $\FunctionalM{B}(\lambda)\le\FunctionalM{A}(\lambda)$, for all $\lambda\le\Lambda$.
Hence, $\RootM{A}\ge\RootM{B}$.

\par
If $M(B)>0$, then there is $B\supseteq A\in\BorelUnitDiam$ with $M(A)>0$.
Because $\FunctionalM{A}(\lambda)=1-\lambda M(A)$, $\RootM{A}=1/M(A)$.
Contra-variance yields $\RootM{B}\le\RootM{A}=1/M(A)<\infty$.
The continuity of $\FunctionalM{B}$ from Proposition~\ref{prop_continuity_functional} implies together with $\FunctionalM{B}(0)=\GfM{\emptyset}=1$ that $\RootM{B}>0$.

\par
The continuity of $\FunctionalM{B}$ renders the infimum a minimum.
Contra-variance implies that for all $\lambda\le\RootM{B}$ and $B\supseteq A\in\BorelBounded$, $\FunctionalM{A}(\lambda)\ge 0$.
Proposition~\eqref{prop_ce} implies that $\RootM{B}M|_B\in\MeasuresShearerExists$.
\end{proof}

\begin{Prop}
\label{prop_continuity_space}
Let $B\in\BorelBounded$ be a continuity set of $M$, i.e., $M(\ClosureOf{B}\setminus B)=0$, where $\ClosureOf{B}$ is the closure of $B$.
For each  sequence $(B_n)_{n\in\Naturals}$ in $\BorelBounded$ decreasing to $B$ with $M(B_1)>0$, the limit $\GfM{B_n}\xrightarrow[n\to\infty]{}\GfM{B}$ holds.
\end{Prop}

\begin{proof}
For $A,B\in\BorelBounded$ with $A\subseteq B$ and $M(B)>0$, bound the difference as
\begin{align*}
 \Modulus{\GfM{B}-\GfM{A}} ={}
 &\BigModulus{
  \sum_{n=0}^\infty
   \frac{(-1)^n}{n!}
   \int_{B^n\setminus A^n}
    \IsHCTuple{n}(x_1,\dotsc,x_n)\prod_{i=1}^n M(\D{x_i})
  }
 \\\stackrel{\eqref{eq_int_bound}}{\le}{}
 &\sum_{n=1}^\infty \frac{1}{n!} (M(B)^n-M(A)^n)
 \\={}
 &\sum_{n=1}^\infty \frac{1}{n!}
  \sum_{j=1}^n \binom{n}{j} M(B\setminus A)^{j} M(A)^{n-j}
 \\\le{}
 &\sum_{n=1}^\infty \frac{1}{n!}
  M(B\setminus A) 2^n \max\EnumSet{1,M(B)}^{n-1}
 \\\le{}
 & 2M(B\setminus A)\exp(\max\EnumSet{1,M(B)})
 \,.
\end{align*}
Thus, $\Modulus{\GfM{B_n}-\GfM{B}}
 \le 2 M(B_n\setminus B)\exp(\max\EnumSet{1,M(B_1)})
 \xrightarrow[n\to\infty]{\text{continuity of $M$ at $B$}}0$.
\end{proof}

\subsection{Behaviour under the set difference operator}
\label{sec_gf_diffop}

\par
The \emph{set difference operator} $\DiffOpSymbol$ at $A\in\BorelAll$ transforms a set function $\phi:\BorelAll\to\Reals$ into $\DiffOpSymbol(A)\phi:\BorelAll\to\Reals\quad B\mapsto \phi(B)-\phi(B\cup A)$.

\begin{Lem}
\label{lem_diffop_iterated_canonical}
For all $n\in\Naturals$ and $A_1,\dotsc,A_n,B\in\BorelAll$ and $\phi:\BorelAll\to\Reals$, iterated application of difference operators commutes and has the canonical form
\begin{equation}
\label{eq_diffop_iterated_canonical}
 \DiffOpSymbol(\EnumSet{A_1,\dotsc,A_n}):=
 \DiffOpSymbol(A_1)(\dotso(\DiffOpSymbol(A_n)\phi)\dotso)(B)
 =
 \sum_{I\subseteq\IntRange{n}}
  (-1)^{\Cardinality{I}}
  \phi(B\cup\bigcup_{i\in I} A_i)
 \,.
\end{equation}
In particular, $\DiffOpSymbol(\EnumSet{A})=\DiffOpSymbol(A)$ and $\DiffOpSymbol(\emptyset)$ is the identity.
\end{Lem}

\ArxivOnly{
\begin{proof}
Commutativity follows from the canonical form for $n=2$.
For $n=1$,
\begin{equation*}
 \DiffOpSymbol(\EnumSet{A})\phi(B)
 = \DiffOpSymbol(A)\phi(B)
 =\phi(B)-\phi(B\cup A)\,.
\end{equation*}
Proceed by induction over $n$.
The induction step from $n$ to $n+1$ is
\begin{align*}
 &\DiffOpSymbol(\EnumSet{A_1,\dotsc,A_{n+1}})\phi(B)
 \\={}& \DiffOpSymbol(A_{n+1})(\DiffOpSymbol(\EnumSet{A_1,\dotsc,A_n})\phi)(B)
 \\={}&(\DiffOpSymbol(\EnumSet{A_1,\dotsc,A_n})\phi)(B)
  - (\DiffOpSymbol(\EnumSet{A_1,\dotsc,A_n})\phi)(B\cup A_{n+1})
 \\={}&\sum_{I\subseteq\IntRange{n}} (-1)^{\Cardinality{I}}
  \phi(B\cup\bigcup_{i\in I} A_i)
  -\phi(B\cup\bigcup_{i\in I} A_i \cup A_{n+1})
 \\={}& \sum_{I\subseteq\IntRange{n+1}} (-1)^{\Cardinality{I}}
  \phi(B\cup\bigcup_{i\in I} A_i)
 \,.\qedhere
\end{align*}
\end{proof}
}

\begin{Lem}
\label{lem_diffop_resummation}
Let $\EnumSet{A_i}_{i=1}^n$ be disjoint Borel sets.
For $I\subseteq\IntRange{n}$, let $A_I:=\bigsqcup_{i\in I} A_i$.
For each $\phi:\BorelAll\to\Reals$, one has
$\sum_{I\subseteq\IntRange{n}}
 \DiffOpSymbol(\EnumSet{A_i}_{i\in I})\phi(A_{\IntRange{n}\setminus I})
 =\phi(\emptyset)$.
\end{Lem}

\ArxivOnly{
\begin{proof}
Abbreviate $\bar{\phi}(I):=\phi(A_I)$.
As $A_\emptyset=\emptyset$, so is $\bar{\phi}(\emptyset)=\phi(A_\emptyset)=\phi(\emptyset)$.
\begin{align*}
 \sum_{I\subseteq\IntRange{n}}
  \DiffOpSymbol(\EnumSet{A_i}_{i\in I})\bar{\phi}(\IntRange{n}\setminus I)
 ={}&\sum_{I\subseteq\IntRange{n}}
  \sum_{J\subseteq I}
   (-1)^{\Cardinality{J}} \bar{\phi}((\IntRange{n}\setminus I)\sqcup J)
 \\={}&\sum_{I\subseteq\IntRange{n}}
  \sum_{J\subseteq I}
   (-1)^{\Cardinality{J}+\Cardinality{I}} \bar{\phi}(\IntRange{n}\setminus J)
 \\={}&\sum_{J\subseteq\IntRange{n}}
  (-1)^{\Cardinality{J}} \bar{\phi}(\IntRange{n}\setminus J)
  \sum_{J\subseteq I\subseteq\IntRange{n}} (-1)^{\Cardinality{I}}
 \\={}&\sum_{J=\IntRange{n}} (-1)^{\Cardinality{J}}
  \bar{\phi}(\IntRange{n}\setminus J) (-1)^{\Cardinality{J}}
 \\={}& \bar{\phi}(\emptyset)
 \,.\qedhere
\end{align*}
\end{proof}
}

\begin{Prop}
\label{prop_gf_diff_completeMonotonicity}
If $M\in\MeasuresShearerExists$, then $\GfSymbol$ is \emph{completely monotone}, i.e.,
for each $n\in\NonNegInts$ and all $A_1,\dotsc,A_n,B\in\BorelFinite$, the iterated difference $\DiffOpSymbol(\EnumSet{A_i}_{i=1}^n)\GfM{B}$ is non-negative.
\end{Prop}

\begin{proof}
Proceed by induction over $n$.
For $n=0$, $M\in\MeasuresShearerExists$ implies that $\DiffOpSymbol(\emptyset)\GfM{B} = \GfM{B}\ge 0$.
For $n=1$, the monotonicity in space of $\GfSymbol$ from Proposition~\ref{prop_monotonicity} implies that
$\DiffOpSymbol(\EnumSet{A_1})\GfM{B}
 = \DiffOpSymbol(A_1)\GfM{B}
 = \GfM{B}-\GfM{B\cup A}
 \ge 0$.
The induction step from $n$ to $n+1$ needs more preparation.
For $I\subseteq\IntRange{n}$, let $A_I:=\bigcup_{i\in I} A_i$.
For $x\in A_{n+1}$, let $B^x:=B\setminus\UnitSphere{x}$ and, for $i\in\IntRange{n+1}$, let $A^x_i:=A_i\setminus\UnitSphere{x}$.
In particular, $A^x_{n+1}=A_{n+1}\setminus\UnitSphere{x}=\emptyset$.
As $\DiffOpSymbol(\emptyset)\phi = \phi$, the degree of the iterated difference decreases in
\begin{align*}
 &\DiffOpSymbol(\EnumSet{A_i}_{i=1}^{n+1})\GfM{B}
 \\={}&\DiffOpSymbol(A_{n+1})(\DiffOpSymbol(\EnumSet{A_i}_{i=1}^n)\GfM{B})
 \\={}&\DiffOpSymbol(\EnumSet{A_i}_{i=1}^n)\GfM{B}
  - \DiffOpSymbol(\EnumSet{A_i}_{i=1}^n)\GfM{B\cup A_{n+1}}
 \\\stackrel{\eqref{eq_diffop_iterated_canonical}}{=}{}&\sum_{I\subseteq\IntRange{n}}
  (-1)^{\Cardinality{I}} \GfM{B\cup A_I}
  -
  \sum_{I\subseteq\IntRange{n}}
  (-1)^{\Cardinality{I}} \GfM{B\cup A_I\cup A_{n+1}}
 \\\stackrel{\eqref{eq_fe}}{=}{}& \sum_{I\subseteq\IntRange{n}} (-1)^{\Cardinality{I}}
  \left(
   \GfM{B\cup A_I}
  -\underbrace{\GfM{B\cup A_I\cup A_{n+1}}}
  \right)
 \\={}& \sum_{I\subseteq\IntRange{n}} (-1)^{\Cardinality{I}}
  \left(
   \GfM{B\cup A_I}
   - \overbrace{\GfM{B\cup A_I}
   + \int_{A_{n+1}} \GfM{(B\cup A_I)\setminus\UnitSphere{x}} M(\D{x})}
  \right)
 \\={}& \sum_{I\subseteq\IntRange{n}} (-1)^{\Cardinality{I}}
  \int_{A_{n+1}} \GfM{B^x\cup A^x_I} M(\D{x})
 \\={}&\int_{A_{n+1}}
  \underbrace{
   \DiffOpSymbol(\EnumSet{A^x_i}_{i=1}^n)\GfM{B^x}
  }_{\ge 0\text{ by the induction hypothesis}}
  M(\D{x})
  \ge 0
 \,.\qedhere
\end{align*}
\end{proof}

\begin{Prop}
\label{prop_gf_diff_zeroesOut}
Let $A_1,\dotsc,A_n,B\in\BorelFinite$ be disjoint.
If $\UPartSymbol\left(B\sqcup\bigsqcup_{i=1}^n A_i\right) < n$, then $\DiffOpSymbol(\EnumSet{A_i}_{i=1}^n)\GfM{B} = 0$.
\end{Prop}

\begin{proof}
For $I\subseteq\IntRange{n}$, let $A_I:=\bigsqcup_{i\in I} A_i$.
Let $A:=A_{\IntRange{n}}$.

\begin{align*}
 &\DiffOpSymbol(\EnumSet{A_i}_{i=1}^n)\GfM{B}
 \\={}
 &\sum_{I\subseteq\IntRange{n}} (-1)^{\Cardinality{I}} \GfM{B\sqcup A_I}
 \\={}
 &\sum_{I\subseteq\IntRange{n}} (-1)^{\Cardinality{I}}
  \sum_{m=0}^\infty \frac{(-1)^m}{m!}
  \int_{(B\sqcup A_I)^m} \IsHCTuple{m}(x_1,\dotsc,x_m)\prod_{i=1}^m M(\D{x_i})
\intertext{
For $m\in\NonNegInts$ and $x_1,\dotsc,x_m\in A\sqcup B$, regard the indices $I(x_1,\dotsc,x_m):=\DescSet{i\in\IntRange{n}}{ \EnumSet{x_1,\dotsc,x_m}\cap A_i\not=\emptyset}$ of the partition elements containing the points.
}
 ={}
 &\sum_{m=0}^\infty \frac{(-1)^m}{m!}
  \int_{(B\sqcup A_I)^m} \IsHCTuple{m}(x_1,\dotsc,x_m)
   \sum_{I(x_1,\dotsc,x_m)\subseteq I\subseteq\IntRange{n}}
    (-1)^{\Cardinality{I}}
   \prod_{i=1}^m M(\D{x_i})
 \\={}
 &\sum_{m=0}^\infty \frac{(-1)^m}{m!}
  \int_{(B\sqcup A_I)^m}
   \IsHCTuple{m}(x_1,\dotsc,x_m)
   \sum_{I(x_1,\dotsc,x_m)=\IntRange{n}} (-1)^n
   \prod_{i=1}^m M(\D{x_i})
\end{align*}
Using~\eqref{eq_hardcoreConfUPartBound} yields $n\le m\le \UPartNumOf{A\sqcup B} < n$ and a zero integrand.
\end{proof}

\begin{Prop}
\label{prop_gf_diff_sumOne}
Let $\EnumSet{A_i}_{i=1}^n$ be disjoint elements of $\BorelUnitDiam$.
For $I\in\IntRange{n}$, let $A_I:=\bigsqcup_{i\in I} A_i$.
For all $r\ge\UPartNumOf{A_{\IntRange{n}}}$,
$\displaystyle
 \sum_{I\subseteq\IntRange{n},\Cardinality{I}\le r}
 \DiffOpSymbol(\EnumSet{A_i}_{i\in I})\GfM{A_{\IntRange{n}\setminus I}} = 1$ holds.
\end{Prop}

\begin{proof}
Fill up with Proposition~\ref{prop_gf_diff_zeroesOut} and evaluate the sum with Lemma~\ref{lem_diffop_resummation}.
\end{proof}

\section{Proofs around Shearer's point process}
\label{sec_spp_proofs}

Section~\ref{sec_spp_char_uniq_exist} contains the uniqueness and characterisation in Proposition~\ref{prop_spp_characterisation_uniqueness} and the existence in Proposition~\ref{prop_spp_existence}.
Together they imply Theorem~\ref{thm_spp}.
Section~\ref{sec_spp_min} proves the minimality of Shearer's PP for the avoidance function in Theorem~\ref{thm_minimality}.
Section~\ref{sec_monotonicity} discusses intrinsic couplings between Shearer's PP at different intensities and the monotonicity properties of the sets of measures.
Section~\ref{sec_spp_different} shows that Shearer's PP law differs from well known hard-core or $1$-dependent PP laws and references probabilistic constructions.
The notation of Section~\ref{sec_gf_properties} applies.

\par
Besides the $\sigma$-algebra of all Borel sets $\BorelAll$, there are also the algebras of bounded, $\UPartSymbol$-finite and less than unit-diameter Borel sets $\BorelBounded$, $\BorelFinite$ and $\BorelUnitDiam$ respectively.
$\BorelFinite$ is an algebra of $\BorelUnitDiam$.
Both algebras $\BorelBounded$ and $\BorelFinite$ generate the $\sigma$-algebra $\BorelAll$.
The distinction between $\BorelFinite$ and $\BorelBounded$ plays a key role, as key proofs employ induction over $\UPartSymbol$.
The strategy is to first establish results on $\BorelFinite$, extend them by $\sigma$-finiteness to $\BorelAll$ and and project them onto $\BorelBounded$.
Sufficient conditions for $\BorelFinite$ and $\BorelBounded$ to coincide are that $(\Space,\Metric)$ is either $\sigma$-compact or total (bounded sets are pre-compact).
The structure of the space below the distance $1$ plays no role.

\subsection{Characterisation, uniqueness and existence}
\label{sec_spp_char_uniq_exist}

\par
For $a,n\in\NonNegInts$, let the \emph{falling factorial} be $\Pochhammer{a}{n} := \prod_{i\in\IntRange{n}} (a-i+1)$.
The \emph{factorial moment measure} of a PP $\xi$ of order $n$ on $B\in\BorelAll$ is $\Expect(\Pochhammer{\xi(B)}{n})$~\cite[Section 9.5]{Daley_VereJones__AnIntroductionToTheTheoryOfPointProcesses_II__Springer_2008}.

\begin{Prop}
\label{prop_spp_characterisation_uniqueness}
If there exists a strong $1$-dependent and $1$-hard-core PP $\ShearersPPM$ with intensity measure $M\in\MeasuresBounded$, then its factorial moment measure of order $n$ at $B\in\BorelBounded$ fulfils
$\displaystyle
 \Expect(\Pochhammer{\ShearersPPM(B)}{n})
 = \int_{B^n} \IsHCTuple{n}(x_1,\dotsc,x_n) \prod_{i=1}^n M(\D{x_i})
$ and its avoidance function is $\GfSymbol$.
There is at most one PP with these three properties in $\OneDepPPM$.
\end{Prop}

\begin{proof}
\par
If $\ShearersPPM$ has finite factorial moment measures of all orders, then its avoidance function is
$\displaystyle
 \Proba(\ShearersPPM(B)=0)
 = \sum_{n=0}^\infty \frac{(-1)^n}{n!} \Expect(\Pochhammer{\xi(B)}{n})
 = \GfM{B}$~\cite[(5.4.10)]{Daley_VereJones__AnIntroductionToTheTheoryOfPointProcesses_I__Springer_2003}.
Because the avoidance function determines the PP's law~\cite[9.2.XIII]{Daley_VereJones__AnIntroductionToTheTheoryOfPointProcesses_II__Springer_2008}, uniqueness follows.

\par
For all $r,n_1,\dotsc,n_r\in\NonNegInts$ with $n:= \sum_{i=1}^r n_i$ and disjoint $A_1,\dotsc,A_r\in\BorelUnitDiam$, show that
\begin{equation}
\label{eq_spp_factor_block}
 \Expect\left(\prod_{i=1}^r \Pochhammer{\ShearersPPM(A_i)}{n_i}\right)
 = \int_{\prod_{i=1}^r A_i^{n_i}}
  \IsHCTuple{n}(x_1,\dotsc,x_n)
  \prod_{l=1}^{n} M(\D{x_l})
 \,.
\end{equation}
The $1$-hard-core of $\ShearersPPM$ and~\eqref{eq_hardcoreConfUPartBound} imply that $\Pochhammer{\ShearersPPM(A_i)}{n_i} = 1$, if both $\ShearersPPM(A_i)=1$ and $n_i=1$, and $0$ else.
Suppose that there is an $j\in\IntRange{r}$ with $n_j\ge 2$.
On the one hand, the $j$-th factor of the lhs of~\eqref{eq_spp_factor_block} equals $0$, whence the lhs of~\eqref{eq_spp_factor_block} equals zero.
On the other hand, this gives an upper bound on the rhs of~\eqref{eq_spp_factor_block} of
\begin{equation*}
 \left(\prod_{j\not=i\in\IntRange{r}} M(A_i)^{n_i}\right)
 \int_{A_j^{n_j}}
  \underbrace{\IsHCTuple{n_j}(x_1,\dotsc,x_{n_j})}_{=0}
  \prod_{l=1}^{n_j} M(\D{x_l})
 = 0
 \,.
\end{equation*}
The remaining case has $r=n$ and all $n_i=1$.
Proceed by induction over $n$.
If $n=1$ and $A\in\BorelUnitDiam$, then
$
 \Expect(\ShearersPPM(A))
 =\Proba(\ShearersPPM(A)=1)=M(A)
 =\int_{A}\IsHCTuple{1}(x)M(\D{x})
$.
The induction step from $(n-1)$ to $n$ is
\begin{align*}
 \Expect\left(\prod_{i=1}^n \ShearersPPM(A_i)\right)
 \stackrel{\eqref{eq_hardcoreConfUPartBound}}{=}{}
 &\Proba(\forall i\in\IntRange{n}:\ShearersPPM(A_i)=1)
 \\={}
 &\int_{A_n}
  \Proba_{x_n}(\forall i\in\IntRange{n-1}:\ShearersPPM(A_i)=1) M(\D{x_n})
 \\\stackrel{\eqref{eq_onedep_weak}}{=}{}
 &\int_{A_n} \Proba(
   \forall i\in\IntRange{n-1}:\ShearersPPM(A_i\setminus\UnitSphere{x_n})=1
  ) M(\D{x_n})
 \\={}
 &\int_{A_n} \left(
   \int_{\prod_{i=1}^{n-1} (A_i\setminus\UnitSphere{x_n})}
   \IsHCTuple{n-1}(x_1,\dotsc,x_{n-1}) \prod_{i=1}^{n-1} M(\D{x_i})
  \right)
  M(\D{x_n})
 \\={}
 &\int_{A_n} \left(
   \int_{\prod_{i=1}^{n-1} A_i}
   \IsHCTuple{n}(x_1,\dotsc,x_{n}) \prod_{i=1}^{n-1} M(\D{x_i})
  \right)
  M(\D{x_n})
 \\={}
 &\int_{\prod_{i=1}^{n} A_i}
   \IsHCTuple{n}(x_1,\dotsc,x_{n})
   \prod_{i=1}^{n} M(\D{x_i})
 \,.
\end{align*}

Let $k:=\UPartNumOf{B}$ and $\EnumSet{A_i}_{i=1}^k$ be a, possibly countable, partition of $B$ into elements of $\BorelUnitDiam$.
If $k=\infty$, then $\IntRange{k}=\Naturals$.
Let $\MNCVISets:=\DescSet{(n_1,\dotsc,n_k)\in\NonNegInts^k}{\sum_{i=1}^k n_i = n}$.
For $a\in\Reals^k$, the multinomial Chu-Vandermonde identity is
\begin{equation*}
 \Pochhammer{\left(\sum_{i\in\IntRange{k}} a_i\right)}{n}
 = \sum_{(n_1,\dotsc,n_k)\in\MNCVISets}
   \binom{n}{n_1,\dotsc,n_k}
   \prod_{i=1}^k \Pochhammer{a_i}{n_i}
 \,.
\end{equation*}
The multinomial Chu-Vandermonde identity together with~\eqref{eq_spp_factor_block} yields
\begin{align*}
 \Expect(\Pochhammer{\ShearersPPM(B)}{n})
 ={}
 &\Expect\left(
   \sum_{(n_1,\dotsc,n_k)\in\MNCVISets}
    \binom{n}{n_1,\dotsc,n_k} \prod_{i=1}^k \Pochhammer{\ShearersPPM(A_i)}{n_i}
  \right)
 \\={}
 &\sum_{(n_1,\dotsc,n_k)\in\MNCVISets}
   \binom{n}{n_1,\dotsc,n_k}
   \Expect\left(
    \prod_{i=1}^k \Pochhammer{\ShearersPPM(A_i)}{n_i}
   \right)
 \\={}
 &\sum_{(n_1,\dotsc,n_k)\in\MNCVISets}
   \binom{n}{n_1,\dotsc,n_k}
   \int_{\prod_{i=1}^k A_i^{n_i}}
    \IsHCTuple{n}(x_1,\dotsc,x_{n})
    \prod_{l=1}^{n} M(\D{x_l})
 \\={}
 &\int_{B^n}
   \IsHCTuple{n}(x_1,\dotsc,x_{n})
   \prod_{l=1}^{n} M(\D{x_l})
 \,.\qedhere
\end{align*}
\end{proof}

\begin{Prop}
\label{prop_spp_existence}
If $M\in\MeasuresShearerExists$, then a strong $1$-dependent and $1$-hard-core PP with intensity $M$ exists.
\end{Prop}

\begin{proof}
Four sufficient conditions~\cite[9.2.XV]{Daley_VereJones__AnIntroductionToTheTheoryOfPointProcesses_II__Springer_2008} jointly guarantee the existence of a simple PP with avoidance function $\GfSymbol$.
The conditions are
\begin{enumerate}
\item
Proposition~\ref{prop_gf_diff_completeMonotonicity} implies the complete monotonicity of $\GfSymbol$.
\item
Unit at $\emptyset$, i.e., $\GfM{\emptyset}=1$, holds trivially.
\item
Continuity in space at $\emptyset$ follows from Proposition~\ref{prop_continuity_space} combined with the fact that $\emptyset$ is a continuity set of $M$, as $M(\ClosureOf{\emptyset}\setminus\emptyset)=M(\emptyset)=0$.
\item
Almost-sure bounded finiteness of the PP.
Let $\EnumSet{A_{i,n}}_{i\in\IntRange{k_n},n\in\Naturals}$ be a \emph{dissecting system}~\cite[Prop A2.1.V]{Daley_VereJones__AnIntroductionToTheTheoryOfPointProcesses_I__Springer_2003} of $B\in\BorelFinite$.
By intersecting every partition of the dissecting system with a fixed $\BorelUnitDiam$-partition of $B$, one may assume that the dissecting systems contains only finite $\BorelUnitDiam$-partitions.
Let
\begin{equation*}
 F(n,r):=\sum_{I\subseteq\IntRange{k_n},\Cardinality{I}\le r}
  \DiffOpSymbol(\EnumSet{A_{i,n}}_{i\in I})
  \GfM{B\setminus\bigsqcup_{i\in I} A_{i,n}}
  \,.
\end{equation*}
By Proposition~\ref{prop_gf_diff_sumOne}, $F(n,r)=1$ for $n\ge r\ge\UPartNumOf{B}$.
Hence, $\displaystyle\lim_{r\to\infty}\lim_{n\to\infty} F(n,r) = 1$.
\end{enumerate}

\par
Recall that the algebra $\BorelFinite$ generates the $\sigma$-algebra $\BorelAll$.
Thus, there exists a simple PP $\ShearersPPM$ on $\Space$ with avoidance function $\GfSymbol$ on $\BorelFinite$.
It rests to show, that the PP $\ShearersPPM$ is simple, strong $1$-dependent, $1$-hard-core and has intensity measure $M$.
The characterisation in Proposition~\ref{prop_spp_characterisation_uniqueness} shows that there is a unique extension of its law to all of $\BorelBounded$.
\par
Simpleness follows from the $1$-hard-core.
By Proposition~\ref{prop_gf_basic_properties}, the function $\GfSymbol$ is $1$-multiplicative.
Proposition~\ref{prop_strongOneMulti} asserts strong $1$-dependence.
\par
Let $A\in\BorelUnitDiam$.
The PP $\varphi_\ClosureOf{A}$ on the closure $\ClosureOf{A}$ of $A$ chooses no point with probability $1-M(\ClosureOf{A})$ and one point with probability $M(\ClosureOf{A})$ distributed with the density $M(\D{x})/M(\ClosureOf{A})$.
The avoidance functions of $\varphi_\ClosureOf{A}$ and $\ClosureOf{A}\ShearersPPM$ coincide, because, for each $A\supseteq B\in\BorelBounded$,
\begin{align*}
 \Proba(\varphi_\ClosureOf{A}(B)=0)
 ={}& 1-M(\ClosureOf{A}) + M(\ClosureOf{A})
  \int_{\ClosureOf{A}\setminus B} \frac{M(\D{x})}{M(\ClosureOf{A})}
 \\={}& 1-M(\ClosureOf{A}) + M(\ClosureOf{A}\setminus B)
 = 1 - M(B)
 = \GfM{B}\,.
\end{align*}
For a a countable dense subset $S$ of $\Space$, consider the following countable subset of $\BorelUnitDiam$.
\begin{equation*}
 \ClosedUnions:=\DescSet{
  \DescSet{x\in\Space}%
          {\Distance{x}{s}\le \alpha\text{ or }\Distance{x}{t}\le\alpha}
  }{
  s,t\in S\text{ with }1-3\alpha:=\Distance{s}{t}<1
  }
 \,.
\end{equation*}
If $\Distance{x}{y}<1$, then there is a closed $A\in\ClosedUnions$ containing both $x$ and $y$.
Therefore,
\begin{align*}
 \Proba(\ShearersPPM\text{ is not $1$-hard-core})
 ={}
 &\Proba(\inf\DescSet{\Distance{x}{y}}{x,y\in\ShearersPPM}<1)
 \\={}
 &\Proba(\exists\, A\in\ClosedUnions: \ShearersPPM(A)\ge 2)
 \le{}
 \sum_{A\in\ClosedUnions}\Proba(\varphi_A\ge 2)
 ={}
 0
 \,.
\end{align*}
\par
For closed $A\in\BorelUnitDiam$ and all $A\supseteq B\in\BorelBounded$,
\begin{equation*}
 \Expect(\ShearersPPM(B))
 =\Expect(A\ShearersPPM(B))
 =\Expect(\varphi_A(B))
 = \int_B \frac{1}{M(A)} M(\D{x}) M(A)
 = M(B)
 \,.
\end{equation*}
Linearity of expectations extends this to the intensity measure of $\ShearersPPM$.
\end{proof}

\subsection{Proof of Theorem~\ref{thm_minimality}}
\label{sec_spp_min}

\par
First, prove the statement only over $\BorelFinite$.
The general statement over $\BorelBounded$ follows by taking limits along sequences in $\BorelFinite$ to a limit in $\BorelBounded$.
Let $A,B\in\BorelFinite$.
Assume $\GfM{B}>0$.
Use induction over $k:=\UPartNumOf{A\cup B}$.
Let $\xi\in\OneDepPPM$ with avoidance function $\AvFun$.
If $\AvProba{B}>0$, let $\AvProbaRatio{A}{B}:=\AvProba{A\cup B}/\AvProba{B}$.
If $k=0$, then $A=B=\emptyset$ and $\AvProbaRatio{\emptyset}{\emptyset}=1=\GfMR{\emptyset}{\emptyset}$.
If $k>0$, then telescoping~\eqref{eq_telescope} restricts to the case $A\in\BorelUnitDiam$ and $A\cap B=\emptyset$.
Let $\EnumSet{A_i}_{i=1}^k$ be a $\BorelUnitDiam$-partition of $A\sqcup B$.
For $x\in A$, let $A(x)$ be the unique partition element containing $x$.
Apply Proposition~\ref{prop_fi} twice to get
\begin{subequations}
\label{eq_proof_min}
\begin{equation}
\label{eq_proof_min_first}
\begin{aligned}
 \AvProbaRatio{A}{B}
 ={}& 1 - \int_{A} \AvProbaRatio{B}{B\setminus\UnitSphere{x}}^{-1} M(\D{x})
 \\={}& 1 - \int_{A}
  \AvProbaRatio{B}{B\setminus A(x)}^{-1}
  \AvProbaRatio{B\setminus A(x)}{B\setminus\UnitSphere{x}}^{-1}
  M(\D{x})
\end{aligned}
\end{equation}
and, for $x\in A$,
\begin{equation}
\label{eq_proof_min_second}
 \AvProbaRatio{B}{B\setminus A(x)}
 = 1 - \int_{B\cap A(x)}
  \AvProbaRatio{B\setminus A(x)}{B\setminus A(x)\setminus\UnitSphere{y}}^{-1}
  M(\D{y})
 \,.
\end{equation}
\end{subequations}
\par
The second half of this proof shows that the expansions~\eqref{eq_proof_min} are well-defined.
For $x\in A$ and $y\in A(x)$, one has $A(x)\subseteq\UnitSphere{y}$ and $\UPartNumOf{B\setminus A(x)}\le k-1$.
Hence, the inductive hypothesis applies to the integrand in~\eqref{eq_proof_min_second} and the second factor of the integrand in~\eqref{eq_proof_min_first}.
For $x\in A$, the inductive hypothesis implies that $\AvProba{B\setminus A(x)}>0$.
Apply~\eqref{eq_fe} to~\eqref{eq_proof_min_second} to see that
\begin{align*}
 \AvProbaRatio{B}{B\setminus A(x)}
 \ge{}
 &1 - \int_{B\cap A(x)}
  \GfMR{B\setminus A(x)}{B\setminus A(x)\setminus\UnitSphere{y}}^{-1} M(\D{y})
 \\={}
 &\GfMR{B}{B\setminus A(x)}
 \,.
\end{align*}
Substitute this into~\eqref{eq_proof_min_first}, multiply and see that this implies that $\AvProba{B}>0$.
Apply~\eqref{eq_fe} to~\eqref{eq_proof_min_first} and obtain
\begin{align*}
 \AvProbaRatio{A}{B}
 \ge{}& 1 - \int_{A}
  \GfMR{B}{B\setminus A(x)}^{-1}
  \GfMR{B\setminus A(x)}{B\setminus\UnitSphere{x}}^{-1}
  M(\D{x})
 \\={}& 1 - \int_{A} \GfMR{B}{B\setminus\UnitSphere{x}}^{-1} M(\D{x})
  =  \GfMR{A}{B}\,.
\end{align*}

\subsection{Proof of Theorem~\ref{thm_nonstrict_implies_zero}}
\label{sec_spp_vanishing}

\par
If $M$ has an atom of mass one at $x$, then the strong $0$-dependent PP $\xi$ from Proposition~\ref{prop_class_empty} has $\Proba(\xi(\EnumSet{x})=0)=0$.
Thus, it suffices to consider only measures with atoms smaller than one.
For every $M\in\MeasuresShearerExists\setminus\MeasuresShearerStrict$, the avoidance function of Shearer's PP $\ShearersPPM$ vanishes on some bounded Borel set and Proposition~\ref{thm_nonstrict_implies_zero} follows trivially.
\par
For every $M\not\in\MeasuresShearerExists$, there exists $B\in\BorelBounded$ with $\GfG{B}{M}<0$ and $M(B)>0$.
Let $\ClosureOf{B}$ be the closure of $B$.
Proposition~\ref{prop_root} implies that $\RootM{B}=\min\DescSet{\lambda}{\GfG{B}{\lambda M}\le 0}$, $0<\RootM{\ClosureOf{B}}\le\RootM{B}<1$ and $\RootM{\ClosureOf{B}}M|_{\ClosureOf{B}}\in\MeasuresShearerExists$.
Proposition~\ref{prop_monotonicity} asserts that, for each $\ClosureOf{B}\supseteq A\in\BorelBounded$ and $N\le\RootM{B}M$, $\GfG{A}{N}\ge 0$.
From here on, assume that $B$ is closed with $\GfG{B}{M}<0$.
Let $\Lambda:=\RootM{B}$.
\par
Consider three independent PPs.
\begin{itemize}
\item
Proposition~\ref{prop_class_empty} guarantees the existence of a strong $1$-dependent $\varphi\in\OneDepPPM$.
\item
For a yet undetermined $N\in\MeasuresBounded$, a (maybe non-simple) Poisson($N$) PP $\vartheta$ on $B$.
\item
As $\Lambda M|_B\in\MeasuresShearerExists$, Proposition~\ref{prop_spp_existence} guarantees the existence Shearer's PP $\ShearersPP{\Lambda M}$.
\end{itemize}
Recall that $\SupportPPOf{\psi}$ is the simple support PP of a general PP $\psi$.
The target PP is $\xi := (\Space\setminus B)\varphi + \SupportPPOf{(\ShearersPP{\Lambda M}+\vartheta)}$.
As all three component PPs are strong $1$-dependent, so is $\xi$.
Because
$\Proba(\xi(B)=0)
 =\Proba(\ShearersPP{\Lambda M}(B)=0,\vartheta(B)=0)
 \le \Proba(\ShearersPP{\Lambda M}(B)=0)
 = 0
$, the avoidance probability of $\xi$ vanishes on $B$.
\par
To determine $N$, verify that it has finite mass on $B$ and that $M$ is the intensity of $\xi$.
Let $\Atoms$ be the atoms of $M$ in $B$.
The diffuse domain is $\DiffuseDomain:=B\setminus\Atoms$.
Set $N|_\DiffuseDomain:=(1-\Lambda)M|_\DiffuseDomain$, which is finite.
Because $\DiffuseDomain\ShearersPP{\Lambda M}$ and $\DiffuseDomain\vartheta$ are simple and independent, for each $B\supseteq A\in\BorelBounded$,
\begin{equation*}
 \Expect(\xi(\DiffuseDomain\cap A))
 = \Expect(\SupportPPOf{(\ShearersPP{\Lambda M}+\vartheta)}(\DiffuseDomain\cap A))
 = \Expect(\ShearersPP{\Lambda M}(\DiffuseDomain\cap A))
    + \Expect(\vartheta(\DiffuseDomain\cap A))
 = M(\DiffuseDomain\cap A)
 \,.
\end{equation*}
For an atom $x\in\Atoms$, let $n_x := \Expect(\vartheta(\EnumSet{x}))$.
The construction demands that
\begin{align*}
 m_x
 :=\Expect(\xi(\EnumSet{x}))
 ={}& \Expect(\SupportPPOf{(\ShearersPP{\Lambda M}+\vartheta)}(\EnumSet{x}))
 \\={}& \Proba(\ShearersPP{\Lambda M}(\EnumSet{x})=1) + \Proba(\ShearersPP{\Lambda M}(\EnumSet{x})=0,\vartheta(\EnumSet{x})\ge 1)
 \\={}& \Expect(\ShearersPP{\Lambda M}(\EnumSet{x}))
   + \GfG{\EnumSet{x}}{\Lambda M}(1-\Proba(\vartheta(\EnumSet{x})=0))
\\={}& \Lambda m_x + (1-m_x)(1-\exp(-n_x))\,.
\end{align*}
Since $m_x<1$, then so is $\frac{(1-\Lambda)m_x}{1-\Lambda m_x}<1$ and
\begin{equation*}
 n_x
 = -\log\left(1-\frac{(1-\Lambda)m_x}{1-\Lambda m_x}\right)
 \stackrel{\eqref{eq_log_inequality}}{\le}\frac%
  {\frac{(1-\Lambda)m_x}{1-\Lambda m_x}}
  {1-\frac{(1-\Lambda)m_x}{1-\Lambda m_x}}
 = \frac{(1-\Lambda) m_x}{1-m_x}\,.
\end{equation*}
As $M(\Atoms)\le M(B)<\infty$ and all atoms of $M$ are less than one, let $m_B:=\max\DescSet{m_x}{ x\in\Atoms}<1$.
The finiteness of the atomic part of $N$ follows from
\begin{equation*}
 N(\Atoms)
 = \sum_{x\in\Atoms} n_x
 \le \sum_{x\in\Atoms} \frac{(1-\Lambda) m_x}{1-m_x}
 \le \frac{(1-\Lambda)}{1-m_B} \sum_{x\in\Atoms} m_x
 = \frac{(1-\Lambda)}{1-m_B} M(\Atoms)
 <\infty\,.
\end{equation*}

\subsection{Intrinsic coupling and monotonicity}
\label{sec_monotonicity}

\begin{Prop}
\label{prop_thinning}
Let $\ShearersPP{M}$ be Shearer's PP with intensity measure $M$.
Let $p:\Space\to[0,1]$ be measurable.
Define $N\in\MeasuresBounded$ by $N(B):=\int_B p(x) M(\D{x})$.
The independent $p$-thinning~\cite[Section 11.3]{Daley_VereJones__AnIntroductionToTheTheoryOfPointProcesses_II__Springer_2008} of $\ShearersPP{M}$ has the same law as $\ShearersPP{N}$.
\end{Prop}

\begin{proof}
Independent thinning preserves strong $1$-dependence and the $1$-hard-core.
It also implies the intensity measure $N$.
Conclude by the uniqueness from Proposition~\ref{prop_spp_characterisation_uniqueness}.
\end{proof}

\begin{proof}[Proof of Theorem~\ref{thm_monotonicity}]
Choose the thinning in Proposition~\ref{prop_thinning} with $p=\frac{\D{N}}{\D{M}}$.
This proves that $\MeasuresShearerExists$ and $\MeasuresShearerStrict=\MeasuresPhasePositive$ are \emph{down-sets}.
The proofs of Proposition~\ref{prop_class_empty} and Theorem~\ref{thm_nonstrict_implies_zero} imply that $\MeasuresEmpty$ and $\MeasuresPhaseZero\sqcup\MeasuresEmpty$ are \emph{up-sets} respectively.
\end{proof}

\subsection{Shearer's PP is different}
\label{sec_spp_different}

\par
Except in trivial cases (zero intensity measure, space with isolated small components), Shearer's PP differs from other well known hard-core PPs.
Because a hard-core radius lower bounds a dependence radius, the difficulty is combining $1$-dependence and $1$-hard-core.
Shearer's PP is not a Poisson PP, one of Matérn's constructions~\cite{Matern__SpatialVariation_StochasticModelsAndTheirApplicationToSomeProblemsInForestSurveysAndOtherSamplingInvestigations__MFSS_1960,Stoyan_Stoyan__OnOneOfMaternsHardcorePointProcessModels__MN_1985,Teichmann_Ballani_Boogaart__GeneralisationsOfMaternsHardcorePointProcesses__SpatStat_2013}, a hard-sphere model as in Section~\ref{sec_hardspheres} and is neither a determinantal nor permanental PP~\cite{Borodin__DeterminantalPointProcesses__OUP_2011,Eisenbaum__StochasticOrderForAlphaPermanentalPointProcesses__SPA_2012}.
In special cases explicit constructions are possible, though.
On the graph $\Integers$ for homogeneous intensity and all radii $R$~\cite{Mathieu_Temmel__KIndependentPercolationOnTrees__SPA_2012}, on $\Reals$ for homogeneous intensity and all radii $R$~\cite{HoferTemmel__ShearersPointProcessAndTheHardSphereModelInOneDimension} and on chordal graphs for $R=2$ and all admissible intensities~\cite{HoferTemmel_Lehner__CliqueTreesOfInfiniteLocallyFiniteChordalGraphs}.
\par
The space $(\Space,\Metric)$ is \emph{$r$-connected}, if each pair of points is part of a finite point sequence with consecutive pairwise distance less than $r$, and $(\Space,\Metric)$ is \emph{$r$-disconnected}, if $\inf\DescSet{\Distance{x}{y}}{\EnumSet{x,y}\subseteq\Space}\ge r$.
If $(\Space,\Metric)$ is $1$-disconnected, then Shearer's PP is a product BRF.
For the remainder of this section, assume that $(\Space,\Metric)$ is $1$-connected with diameter greater than one and that, for all $x\in\Space$ and each neighbourhood $B$ of $x$, $M(B)>0$.

\begin{Prop}
\label{prop_spp_not_matern}
Shearer's PP is not a Matérn-style hard-core PP.
\end{Prop}

\begin{proof}
\par
Let $N\in\MeasuresBounded$.
For the extreme cases of diffuse and atomic $N$, let $\varphi$ be a Poisson($N$) PP and a product BRF of intensity $N$ respectively.
Additionally, attach iid Uniform($[0,1]$) marks to points of $\varphi$.
There are disjoint $A_1,A_2,A_3\in\BorelUnitDiam$ with positive $M$-measure with the diameters of $A_1\sqcup A_2$ and $A_2\sqcup A_3$ smaller than $R$ and $\Distance{A_1}{A_3}\ge R$.
This reduces the setting to BRFs on the graph $G:=(\IntRange{3},\EnumSet{\EnumSet{1,2},\EnumSet{2,3}})$ with the geodesic metric $d$, inducing the metric space $(\IntRange{3},2d)$.
The aim is to show that an $1$-hard-core clashes with independence of the marginals at $1$ and $3$.
Let $(n_1,n_2,n_3)$ and $(m_1,m_2,m_3)$ be the underlying and positive target atomic intensities respectively.
\par
For the Matérn~I hard-core PP~\cite{Matern__SpatialVariation_StochasticModelsAndTheirApplicationToSomeProblemsInForestSurveysAndOtherSamplingInvestigations__MFSS_1960,Stoyan_Stoyan__OnOneOfMaternsHardcorePointProcessModels__MN_1985}, delete all points of $\varphi$ having at least another point at distance less than $1$.
The target intensities are $m_1 = n_1 (1-n_2)$, $m_2 = (1-n_1)n_2(1-n_3)$ and $m_3 = (1-n_2) n_3$.
$1$-dependence demands that $n_1 (1-n_2)^2 n_3 = m_1 m_3 = n_1 (1-n_2) n_3$.
This implies that $n_2=0$ and contradicts $m_2>0$.
\par
To obtain the Matérn~II hard-core PP~\cite{Matern__SpatialVariation_StochasticModelsAndTheirApplicationToSomeProblemsInForestSurveysAndOtherSamplingInvestigations__MFSS_1960,Stoyan_Stoyan__OnOneOfMaternsHardcorePointProcessModels__MN_1985}, delete every point $x$ of $\varphi$ whose mark $l$ fulfils $l=\max\DescSet{h}{ (y,h)\in\varphi, \Distance{x}{y}<1}$.
By symmetry, the comparison between the labels on neighbouring sites yields a probability of $1/2$ in a site's favour.
The target intensities are $m_1 = n_1 (1-n_2/2)$, $m_2 = n_2 (1 - n_1 /4)(1-n_3 /4)$ and $m_3 = n_3 (1-n_2/2)$.
$1$-dependence demands that $n_1 (1-n_2/2)^2 n_3 = m_1 m_3 = n_1 (1-3n_2/4) n_3$.
This implies that $n_2=0$, a contradiction to $m_2>0$.
\par
A marked point $(x,l)$ inhibits a marked point $(y,k)$, if $\Distance{x}{y}<1$ and $l\le k$.
A marked point $(x,l)$ is uninhibited, if it fulfils $l=\min\DescSet{h}{ (y,h)\in\varphi, \Distance{x}{y}<1}$.
Iteratively, delete all inhibited points.
Uninhibited points only contribute once to the deletion and every $1$-connected cluster of points contains at least one uninhibited point.
Hence, the deletion procedure stabilises locally almost-surely and the resulting PP is the Matérn~III hard-core PP~\cite{Teichmann_Ballani_Boogaart__GeneralisationsOfMaternsHardcorePointProcesses__SpatStat_2013} with radius $1$.
The target intensities are $m_1 = n_1 (1-n_2/2+n_2 n_3/4)$ and $m_3 = n_3 (1-n_2/2+n_2 n_1/4)$.
This implies that $n_1,n_3\in]0,1[$.
$1$-dependence demands that
$n_1 n_3 (1-n_2/2+n_2 n_3/4) (1-n_2/2+n_2 n_1/4)
 = m_1 m_3
 = n_1 (1-n_2 /4) n_3
$.
This is impossible, as the lhs is always bigger than the rhs.
\end{proof}

\begin{Prop}
\label{prop_spp_not_hs}
Shearer's PP is not the hard-sphere model.
\end{Prop}

\begin{proof}
\par
Let $A,B\in\BorelUnitDiam$ with $A\subsetneq B$ and $0<M(A)<M(B)$.
Let $h_B$ be the hard-sphere model with fugacity $N$ on $B$.
It has avoidance function $N/(1+N(B))$.
Demanding equal laws for Shearer's PP of intensity $N$ and the hard-sphere model with fugacity $N$ on $A$, $B$ and $B\setminus A$ leads to only trivial solutions of
\begin{equation*}
 \frac{N(B)}{1+N(B)}
 = M(B)
 = M(A) + M(B\setminus{}A)
 = \frac{N(A)}{1+N(A)}+\frac{N(B)-N(A)}{1+M(B)-M(A)}
 \,.
\end{equation*}
\end{proof}

\begin{Prop}
\label{prop_spp_not_det_perm}
Shearer's PP is neither determinantal nor permanental.
\end{Prop}

\begin{proof}
\par
The higher moment densities of a determinantal PP are the determinants of a matrix with entries from a bivariate, symmetric and measurable kernel $K: \Space^2\to\Reals$.
Consider the correlation function of $n$ points, i.e.
the Radon-Nikodyn derivative of the $n$-th factorial moment measure of $\ShearersPPM$ with respect to the $n$-fold product of $M$.
It depends only on the $1$-connected graph structure of the $n$ points and takes values $1$ and $0$, for $1$-disconnected graphs and graphs containing at least one $1$-edge, respectively.
\par
For $x\in\Space$, this yields $\det K(x,x) = 1$.
For $\EnumSet{x,y}\subseteq\Space$, this yields $K(x,y)=\pm 1$ and $K(x,y)=0$ for  $\Distance{x}{y}<1$ and $\Distance{x}{y}\ge 1$ respectively.
For $n=3$, the graph $(\EnumSet{x,y,z},\EnumSet{\EnumSet{x,y},\EnumSet{y,z}})$ yields the contradiction $0 = 1 - K(x,y)^2 - K(y,z)^2 = -1$.

\par
The attraction of permanental PPs contradicts the $1$-hard-core of Shearer's PP.
\end{proof}

\section{Proofs of the LLLs}
\label{sec_lll_proofs}

\begin{Lem}
\label{lem_upartbound_equivalence}
If~\eqref{eq_upartnum_finite} holds, then $\BorelBounded=\BorelFinite$.
\end{Lem}

\begin{proof}
Induction on $\diam(B):=\sup\DescSet{\Distance{x}{y}}{ x,y\in B}$ shows that, for each $B\in\BorelBounded$, $\UPartNumOf{B}\le\UPartSup^{\lfloor\diam(B)\rfloor}<\infty$.
\end{proof}

\begin{proof}[Proof of Theorem~\ref{thm_lll_rough}]
\par
By~\eqref{eq_upartnum_finite} and Lemma~\ref{lem_upartbound_equivalence}, $\BorelBounded=\BorelFinite$.
Let $A,B\in\BorelFinite$.
This proof uses induction over $k:=\UPartNumOf{A\cup B}$.
If $k=0$, then $A=B=\emptyset$ and $\GfMR{\emptyset}{\emptyset}=1$.
If $k>0$, then telescope~\eqref{eq_telescope} to restrict to $A\setminus B\in\BorelUnitDiam$.
Let $\EnumSet{A_i}_{i=1}^k$ be a $\BorelUnitDiam$-partition of $A\cup B$.
For $x\in A\cup B$, let $A(x)$ be the partition element containing $x$.
Apply~\eqref{eq_fe} twice to get
\begin{subequations}
\label{eq_proof_lll_fe}
\begin{equation}
\label{eq_proof_lll_fe_first}
\begin{aligned}
 \GfMR{A}{B}
 ={}& 1 - \int_{A\setminus B} \GfMR{B}{B\setminus\UnitSphere{x}}^{-1} M(\D{x})
 \\={}& 1 - \int_{A\setminus B}
  \GfMR{B}{B\setminus A(x)}^{-1}
  \GfMR{B\setminus A(x)}{B\setminus\UnitSphere{x}}^{-1}
  M(\D{x})
\end{aligned}
\end{equation}
and, for $x\in A\setminus B$,
\begin{equation}
\label{eq_proof_lll_fe_second}
\begin{aligned}
 \GfMR{B}{B\setminus A(x)}
 = 1 - \int_{B\cap A(x)}
  \GfMR{B\setminus A(x)}{B\setminus\UnitSphere{y}}^{-1}
  M(\D{y})
 \,.
\end{aligned}
\end{equation}
\end{subequations}
\par
For $x\in A\setminus B$ and $y\in A(x)$, $A(x)\subseteq\UnitSphere{y}$, whence $\UPartNumOf{B\setminus\UnitSphere{y}}\le\UPartNumOf{B\setminus A(x)}\le k-1$.
Thus, the inductive hypothesis applies to the integrand of~\eqref{eq_proof_lll_fe_second} and the second factor in~\eqref{eq_proof_lll_fe_first}.
Bounding the integrand of~\eqref{eq_proof_lll_fe_second} by the inductive hypothesis~\eqref{eq_lll_rough_bound} leads to
\begin{align*}
 \GfMR{B}{B\setminus A(x)}
 \ge{}
 & 1 - \int_{B\cap A(x)}
  \left(\frac{\UPartSup+2}{\UPartSup+1}\right)^{\UPartNumOf{\UnitSphere{y}}}
  M(\D{y})
 \\\stackrel{\eqref{eq_upartnum_finite}}{\ge}{}
 & 1 - \left(\frac{\UPartSup+2}{\UPartSup+1}\right)^{\UPartSup}
  M(B\cap A(x))
 \\\stackrel{\eqref{eq_lll_rough_measure}}{\ge}
 & 1 - \left(\frac{\UPartSup+2}{\UPartSup+1}\right)^{\UPartSup}
  \frac{(\UPartSup+1)^{\UPartSup+1}}{(\UPartSup+2)^{\UPartSup+2}}
 \\\ge{}
 & 1 - \frac{1}{\UPartSup+2}
  = \frac{\UPartSup+1}{\UPartSup+2}
 \,.
\end{align*}
Substituting this into the rhs of~\eqref{eq_proof_lll_fe_first} and bounding the right factor of the integrand by the inductive hypothesis~\eqref{eq_lll_rough_bound} leads to
\begin{align*}
 \GfMR{A}{B}
 \ge{}
 & 1 - \int_{A\setminus B} \frac{\UPartSup+2}{\UPartSup+1}
 \left(\frac{\UPartSup+2}{\UPartSup+1}\right)^{\UPartNumOf{\UnitSphere{x}}}
 M(\D{x})
 \\\stackrel{\eqref{eq_upartnum_finite}}{\ge}{}
 & 1 - \left(\frac{\UPartSup+2}{\UPartSup+1}\right)^{\UPartSup+1}
  M(A\setminus B)
 \\\stackrel{\eqref{eq_lll_rough_measure}}{\ge}{}
 & 1 - \left(\frac{\UPartSup+2}{\UPartSup+1}\right)^{\UPartSup+1}
  \frac{(\UPartSup+1)^{\UPartSup+1}}{(\UPartSup+2)^{\UPartSup+2}}
 \\={}
 & 1 - \frac{1}{\UPartSup+2} = \frac{\UPartSup+1}{\UPartSup+2}
 \,.\qedhere
\end{align*}
\end{proof}

\begin{proof}[Proof of Theorem~\ref{thm_lll_generic}]
By~\eqref{eq_upartnum_finite} and Lemma~\ref{lem_upartbound_equivalence}, $\BorelBounded=\BorelFinite$.
Theorem~\ref{thm_lll_generic} follows from telescoping the statement $\forall A\in\BorelUnitDiam, B\in\BorelBounded:\GfMR{A}{B}\ge\frac{\alpha}{1+\alpha}$.
Let $N:=(1+\alpha)M$.
The equality~\eqref{eq_fe} and the monotonicity of $\GfRSymbol$ in $M$ from Proposition~\ref{prop_monotonicity} imply that
\begin{align*}
 0\le \GfR{A}{B}{N}
 ={}& 1-\int_{A\setminus B} \GfR{B\setminus A}{B\setminus\UnitSphere{x}}{N}^{-1} N(\D{x})
 \\\le{}& 1-\int_{A\setminus B}
  \GfR{B\setminus A}{B\setminus\UnitSphere{x}}{M}^{-1}\frac{dN}{dM}(x) M(\D{x})
 \\\le{}& 1- (1+\alpha)\int_{A\setminus B}
  \GfR{B\setminus A}{B\setminus\UnitSphere{x}}{M}^{-1} M(\D{x})
 \\={}& 1-(1+\alpha) (1-\GfR{A}{B}{M})
 \,.\qedhere
\end{align*}
\end{proof}

\begin{proof}[Proof of Theorem~\ref{thm_lll_smooth}]
\par
Assuming that~\eqref{eq_lll_smooth_bound} holds for $A,B\in\BorelFinite$, the general case follows from a limiting argument.
Let $A,B\in\BorelBounded$ with $A\cap B=\emptyset$.
Take sequences $(A_n)_{n\in\Naturals}$ and $(B_n)_{n\in\Naturals}$ in $\BorelFinite$ exhausting $A$ and $B$ respectively.
The monotonicity of $\GfRSymbol$ in space from Proposition~\ref{prop_monotonicity} implies that
\begin{equation*}
 \GfMR{A}{B}
 = \lim_{n\to\infty} \GfMR{A_n}{B_n}
 \ge \lim_{n\to\infty} \exp(-N(A_n\setminus B_n))
 = \exp(-N(A\setminus B))\,.
\end{equation*}
\par
It remains to prove~\eqref{eq_lll_smooth_bound} for $A,B\in\BorelFinite$.
The first part of the proof is verbatim the same as the one of Theorem~\ref{thm_lll_rough}, leading to~\eqref{eq_proof_lll_fe}.
For $x\in A$, bounding the integrand in~\eqref{eq_proof_lll_fe_second} by the inductive hypothesis~\eqref{eq_lll_smooth_bound} leads to
\begin{align*}
 \GfMR{B}{B\setminus A(x)}
 \ge{}& 1 - \int_{B\cap A(x)}
  \exp(N((B\setminus A(x))\cap\UnitSphere{y})) M(\D{y})
 \\\ge{}& 1 - \int_{B\cap A(x)}
  \exp(N(\UnitSphere{y}\setminus (B\cap A(x)))) M(\D{y})
 \\\stackrel{\eqref{eq_lll_smooth_condition}}{\ge}{}
 & \exp(-N(B\cap A(x)))
 \,.
\end{align*}
Substituting this into the rhs of~\eqref{eq_proof_lll_fe_first} and using the inductive hypothesis on the right factor of the integrand leads to
\begin{align*}
 \GfMR{A}{B}
 \ge{}& 1 - \int_{A\setminus B}
   \exp(N(B\cap A(x))) \exp(N((B\cup\UnitSphere{x})\setminus A(x))) M(\D{x})
 \\={}& 1 - \int_{A\setminus B}
  \exp(N(B\cap\UnitSphere{x})) M(\D{x})
 \\\ge{}& 1 - \int_{A\setminus B}
  \exp(N(\UnitSphere{x}\setminus(A\setminus B))) M(\D{x})
 \\\stackrel{\eqref{eq_lll_smooth_condition}}{\ge}{}
 & \exp(-N(A\setminus B))
 \,.\qedhere
\end{align*}
\end{proof}

\begin{proof}[Proof of Corollary~\ref{cor_lll_lebesgue}]
Here $X=\Reals^d$, $M=\lambda\Lebesgue$ and $N=\alpha\Lebesgue$.
For each $A\in\BorelUnitDiam$ with $L:=\Lebesgue(A)$, condition~\eqref{eq_lll_smooth_condition_stronger} rewrites into $\lambda L e^{\alpha (V-L)} \le 1-e^{-\alpha L}$.
The identity $\lambda e^{\alpha V}=\alpha$ simplifies this to $\alpha L e^{-\alpha L} + e^{-\alpha L} \le 1$.
The last inequality is just $1+z\le e^z$ with $z:=\alpha L\ge 0$.
Thus, condition~\eqref{eq_lll_smooth_condition} holds.
Conclude via Theorem~\ref{thm_lll_smooth}.
\end{proof}

\section{Additional material}

\subsection{Association and ordering properties}
\label{sec_association_ordering}

\par
This section comments on various ordering properties of Shearer's PP with respect to factorial moment ordering and void function ordering.
It also contrasts the association of the avoidance function of Shearer's PP and the hard-sphere model.
Terminology is taken from~\cite{Blaszczyszyn_Yogeshwaran__ClusteringComparisonOfPointProcessesWithApplicationsToRandomGeometricModels__Springer_2015} or~\cite{Blaszczyszyn_Yogeshwaran__OnComparisonOfClusteringPropertiesOfPointProcesses__AAP_2014}.

\par
Using Proposition~\ref{prop_spp_characterisation_uniqueness} together with the bound~\eqref{eq_int_bound}, implies that all factorial and ordinary moment measures of $\ShearersPPM$ are less than the ones of a Poisson PP of intensity $M$.
In other words, $\ShearersPPM$ is a \emph{$\alpha$-weakly sub Poisson PP}.

\par
For Shearer's PP $\ShearersPPM$, Proposition~\ref{prop_monotonicity} implies that, for disjoint $A,B\in\BorelBounded$,
\begin{subequations}
\label{eq_association}
\begin{equation}
\label{eq_association_spp}
 \Proba(\ShearersPPM(A\sqcup B)=0)
 = \GfG{A\sqcup B}{M}
 \le \Proba(\ShearersPPM(A)=0)\Proba(\ShearersPPM(B)=0)
 \,.
\end{equation}
By~\cite[Proposition 3.1]{Blaszczyszyn_Yogeshwaran__OnComparisonOfClusteringPropertiesOfPointProcesses__AAP_2014}, $\ShearersPPM$ is a \emph{$\nu$-weakly sub Poisson PP}, i.e., $\Proba(\ShearersPPM(B)=0)\le\exp(-M(B))$ holds, for all $B\in\BorelBounded$.
Note: \cite[Proposition 3.1]{Blaszczyszyn_Yogeshwaran__OnComparisonOfClusteringPropertiesOfPointProcesses__AAP_2014} only demands a diffuse measure for the super-weakly Poisson direction.

\par
In the context of Corollary~\ref{cor_lll_lebesgue}, the avoidance function of $\ShearersPP{\lambda\Lebesgue}$ is bigger than the avoidance function of a homogeneous Poisson PP of intensity $\alpha$.
Because $\lambda<\lambda e^{-\lambda V}=\alpha$ (as $V\ge 1$ and assuming that $\lambda\not=0$), $\ShearersPP{\lambda\Lebesgue}$ is a \emph{$\nu$-weakly super Poisson($\alpha$) PP}.
So its avoidance function is bigger than the one of a Poisson intensity of the bigger intensity $\alpha$.

\par
For $k\in\NonNegInts$, let $f:\Reals\to\EnumSet{0,1}$ be the indicator function of zero and $g:=1-f$.
The functions $f$ and $g$ are monotone decreasing and increasing, respectively.
For each PP $\xi$ and $B\in\BorelBounded$, $\Proba(\xi(B)=0)=\Expect(f(\xi(A)))$.
With $a:=\Expect(g(\ShearersPPM(A)))$ and $b:=\Expect(g(\ShearersPPM(B)))$, the inequality~\ref{eq_association_spp} rewrites as
\begin{align*}
 \Proba(\ShearersPPM(A\sqcup B)=0)
 ={}
 & \Expect(f(\ShearersPPM(A\sqcup B)))
 \\={}
 & \Expect(f(\ShearersPPM(A))f(\ShearersPPM(B)))
 \\={}
 & \Expect((1-g(\ShearersPPM(A)))(1-g(\ShearersPPM(B))))
 \\={}
 & 1 - a - b + \Expect(g(\ShearersPPM(A))g(\ShearersPPM(B)))
 \\\le{}
 & 1 - a - b + ab
 \\={}
 & \Expect(1-g(\ShearersPPM(A)))\Expect(1-g(\ShearersPPM(B)))
 \\={}
 & \Expect(f(\ShearersPPM(A)))\Expect(f(\ShearersPPM(B)))
 \\={}
 & \Proba(\ShearersPPM(A)=0)\Proba(\ShearersPPM(B)=0)
 \,.
\end{align*}
This implies that
\begin{equation*}
 \frac%
  {\Expect(g(\ShearersPPM(A))g(\ShearersPPM(B)))}
  {\Expect(g(\ShearersPPM(A)))\Expect(g(\ShearersPPM(B)))}
 \le
 0
 \,.
\end{equation*}
The extension to more than two disjoint sets via Proposition~\ref{prop_monotonicity} is straightforward.
Thus, the avoidance function of $\ShearersPPM$ exhibits \emph{negative association}.

\par
For $A,B\in\BorelBounded$, the spatial submultiplicativity of $\GfG{A\sqcup B}{-M}$ implies that
\begin{equation}
\label{eq_association_hs}
 \Proba(\HardSphereM{A\sqcup B}(A\sqcup B)=0)
 = \frac{1}{\GfG{A\sqcup B}{-M}}
 \ge \Proba(\HardSphereM{A}(A)=0)\Proba(\HardSphereM{B}(B)=0)\,.
\end{equation}
\end{subequations}
Hence, the avoidance function of the hard-sphere model exhibits \emph{positive association}.

\par
The question whether $\ShearersPPM$ is \emph{negatively associated} (cf.~\cite[Sec 2.3]{Blaszczyszyn_Yogeshwaran__OnComparisonOfClusteringPropertiesOfPointProcesses__AAP_2014}) is still open.

\subsection{Miscellaneous proofs}

\par
The classic Chu-Vandermonde identity is
\begin{equation*}
 \Pochhammer{(a+b)}{n}
 = \sum_{n_a=0}^n \binom{n}{n_a} \Pochhammer{a}{n_a}\Pochhammer{b}{n-n_a}
 = \sum_{n_a,n_b} \binom{n}{n_a,n_b} \Pochhammer{a}{n_a}\Pochhammer{b}{n_b}
 \,.
\end{equation*}
The multinomial form in the proof of Proposition~\ref{prop_spp_characterisation_uniqueness} follows by induction.
For $k\ge 3$, one has
\begin{align*}
 \Pochhammer{\left(\sum_{i\in\IntRange{k}} a_i\right)}{n}
 ={}
 &\sum_{n_k=0}^n \binom{n}{n_k} \Pochhammer{a_k}{n_k}
   \Pochhammer{\left(\sum_{i\in\IntRange{k-1}} a_i\right)}{n-n_k}
 \\={}
 &\sum_{n_k=0}^n \binom{n}{n_k} \Pochhammer{a_k}{n_k}
  \sum_{n_1,\dotsc,n_{k-1}}
   \binom{n-n_k}{n_1,\dotsc,n_{k-1}}
   \prod_{i\in\IntRange{k-1}} \Pochhammer{a_i}{n_i}
 \\={}
 &\sum_{n_1,\dotsc,n_k} \binom{n}{n_1,\dotsc,n_k}
   \prod_{i=1}^k \Pochhammer{a_i}{n_i}
 \,.
\end{align*}

\par
The trivial solutions in the proof of Proposition~\ref{prop_spp_not_hs} are as follows.
Let $b:=N(B)$ and $a:=N(A)$.
The condition rewrites to
\begin{equation*}
 \frac{b}{1+b}
 = \frac{a}{1+a} + \frac{b-a}{1+b-a}
 = \frac{a+ab-a^2+b-a+ab-a^2}{(1+a)(1+b-a)}
 = \frac{2ab - 2a^2 +b}{(1+a)(1+b-a)}
 \,.
\end{equation*}
This yields
\begin{align*}
 0 ={}& (2ab - 2a^2 +b)(1+b)-b(1+a)(1+b-a)
   \\={}& -2a^2+b(1+a) + b(2ab-2a^2+b) -b(1+b+ab-a^2)
   \\={}& 2a^2+b(a+ab-a^2)
   \\={}& a(2a+b+b^2+ab)
\end{align*}
Thus, either $a=0$, or $a>0$ and $b^2 + b(1+a)+2a = 0$. The latter equation has no solution with $b\ge 0$.

\par
The determinants used in the proof of Proposition~\ref{prop_spp_not_det_perm} are as follows.
For $n=1$ and each $x\in\Space$, this implies that $\det K(x,x) = 1$.
For $n=2$ and all $x,y\in\Space$ with $\Distance{x}{y}<1$, this yields $0=1-K(x,y)^2$ and $K(x,y)=\pm 1$.
For $n=2$ and all $x,y\in\Space$ with $\Distance{x}{y}\ge 1$, this yields $1=1-K(x,y)^2$ and $K(x,y)=0$.
For $n=3$, the graph $(\EnumSet{x,y,z},\EnumSet{\EnumSet{x,y},\EnumSet{y,z}})$ yields the contradiction
\begin{equation*}
 0 =
 \begin{vmatrix}
  1 & K(x,y) & 0
  \\
  K(x,y) & 1 & K(z,y)
  \\
  0 & K(z,y) & 1
 \end{vmatrix}
 = 1 - K(x,y)^2 - K(y,z)^2
 = -1
 \,.
\end{equation*}

\subsection{Direct proof of homogeneous LLL in Euclidean space}

\begin{proof}[Direct proof of Corollary~\ref{cor_lll_lebesgue}]
Because $\Space=\Reals^d$ satisfies~\eqref{eq_upartnum_finite}, Lemma~\ref{lem_upartbound_equivalence} asserts that $\BorelBounded=\BorelFinite$. Thus, the proof continues from the proof of Theorem~\ref{thm_lll_rough} at~\eqref{eq_proof_lll_fe}. Recall that $M=\lambda\Lebesgue$. For $x\in A\setminus B$, let $l_x:=\Lebesgue(B\cap A(x))$. Bounding the integrand of~\eqref{eq_proof_lll_fe_second} by the inductive hypothesis leads to
\begin{align*}
 \GfMR{B}{B\setminus A(x)}
 \ge{}& 1 - \int_{B\cap A(x)}
  \exp(\alpha\Lebesgue(B\setminus A(x))\cap\UnitSphere{y}))
  \,\lambda\Lebesgue(dy)
 \\\ge{}& 1 - \int_{B\cap A(x)}
  \exp(\alpha\Lebesgue(\UnitSphere{y}\setminus (B\cap A(x))))
  \,\lambda\Lebesgue(dy)
 \\={}& 1 - \lambda l_x \exp(\alpha(V-l_x))
  \qquad\text{ as }\lambda\exp(\alpha V)=\alpha
 \\={}& 1 - \alpha l_x \exp(-\alpha l_x)
 \\\ge{}& \exp(-\alpha l_x)
  \qquad\qquad\qquad\quad\,\,\!\text{ as }e^z\ge 1+z
 \\={}& \exp(-\alpha\Lebesgue(B\cap A(x)))
 \,.
\end{align*}
Let $l:=\Lebesgue(A\setminus B)$. Substituting this into the rhs of~\eqref{eq_proof_lll_fe_first} and using the inductive hypothesis on the right factor of the integrand leads to
\begin{align*}
 \GfMR{A}{B}
 \ge{}& 1 - \int_{A\setminus B}
   \exp(\alpha\Lebesgue(B\cap A(x)))
   \exp(\alpha\Lebesgue((B\setminus A(x))\cap \UnitSphere{x}))
   \,\lambda\Lebesgue(dx)
 \\\ge{}& 1 - \int_{A\setminus B}
  \exp(\alpha\Lebesgue(\UnitSphere{x}\setminus(A\setminus B)))
  \,\lambda\Lebesgue(dx)
 \\={}& 1 - \lambda l \exp(\alpha(V-l))
  \qquad\text{ as }\lambda\exp(\alpha V)=\alpha
 \\={}& 1 - \alpha l \exp(-\alpha l)
 \\\ge{}& \exp(-\alpha l)
  \qquad\qquad\qquad\quad\!\!\text{ as }e^z\ge 1+z
 \\={}& \exp(-\alpha\Lebesgue(A\setminus B))
 \,.\qedhere
\end{align*}
\end{proof}

\subsection{Inductive proofs}

\begin{Prop}
\label{prop_gfr_bound_upper}
If $A\in\BorelUnitDiam$, then $\GfMR{A}{B}\le 1-M(A)$.
\end{Prop}

\begin{proof} The FE~\eqref{eq_fe}, monotonicity~\eqref{eq_fer_monotonicity_space} and the upper bound $1=\GfM{\emptyset}$ yield
\begin{equation*}
 \GfMR{A}{B}
 =1-\int_A
  \GfMR{
   B\cap\UnitSphere{x}}{
   B\setminus\UnitSphere{x}
  }^{-1} M(dx)
 \le 1 - \int_A 1 M(x)
 = 1 - M(A)
 \,.\qedhere
\end{equation*}
\end{proof}

\begin{Prop}
\label{prop_gfr_bound_lower}
If $A_1,A_2\in\BorelUnitDiam$ are disjoint and $B\in\BorelFinite$ with $(A_1\cup B)\cap A_2=\emptyset$ and $(A_1\sqcup A_2)\in\BorelUnitDiam$, then $\GfMR{A_1}{B}\ge M(A_2)$.
\end{Prop}

\begin{proof} Apply the FE~\eqref{eq_fe}, factorize and bound by $1$ from Proposition~\ref{prop_gfr_bound_upper}:
\begin{align*}
 0\le\GfMR{A_2}{A_1\cup B}
 ={}&1-\int_{A_2}
  \GfMR{A_1\cup B}{(A_1\cup B)\setminus\UnitSphere{x}}^{-1} M(dx)
 \\={}&1-\int_{A_2}
  \GfMR{A_1\cup B}{B\setminus\UnitSphere{x}}^{-1} M(dx)
 \\={}&1-\int_{A_2}
  \GfMR{A_1}{B}^{-1} \GfMR{B}{B\setminus\UnitSphere{x}}^{-1} M(dx)
 \\\le{}&1-\int_{A_2} \GfMR{A_1}{B}^{-1}M(dx)
 \\={}& 1 - \GfMR{A_1}{B}^{-1} M(A_2)
 \,.\qedhere
\end{align*}
\end{proof}

\begin{Prop}
\label{prop_gf_worstcase_positivity}
For all $A,B\in\BorelFinite$ with $A\subseteq B$, if
\begin{equation}
\label{eq_fe_worstcase_positivity}
 \GfM{B}>0 \Then \GfM{A}>0\,.
\end{equation}
\end{Prop}

\begin{proof}
Use induction over $k:=\UPartNumOf{B\setminus A}$. If $k=0$, then $A=B$ and $\GfM{A}=\GfM{B}>0$. If $k=1$, then apply the FE~\eqref{eq_fe} to obtain
\begin{equation*}
 \GfM{B\setminus A}
 = \GfM{B} + \int_{B\setminus A} \GfM{B\setminus\UnitSphere{x}} M(dx)
 \ge \GfM{B}
 > 0\,,
\end{equation*}
as $M\in\MeasuresShearerExists$ implies that the integrand is non-negative. If $k>1$, then choose $\EnumSet{A_i}_{i=0}^k$ with $A=:A_0\subsetneq A_1\subsetneq\dotso\subsetneq A_k:=B$ and $\UPartNumOf{A_i\setminus A_{i-1}}=1$, for all $i\in\IntRange{k}$. Apply the statement for the previous case $k$ times and obtain
\begin{equation*}
 \GfM{A}=\GfM{A_0}\ge\GfM{A_1}\ge\dotso\ge\GfM{A_k}=\GfM{B}>0
 \,.\qedhere
\end{equation*}
\end{proof}

\begin{Prop}
\label{prop_gf_monotonicity_space}
For every $A,B\in\BorelFinite$,
\begin{equation}
\label{eq_fe_monotonicity_space}
 A\subseteq B \Then \GfM{A}\ge\GfM{B}\,.
\end{equation}
\end{Prop}

\begin{proof}
Use induction over $k:=\UPartNumOf{B\setminus A}$. If $k=0$, then $A=B$ and $\GfM{A}=\GfM{B}$. If $k=1$, then apply the FE~\eqref{eq_fe} to obtain
\begin{equation*}
 \GfM{B\setminus A}
 = \GfM{B} + \int_{B\setminus A} \GfM{B\setminus\UnitSphere{x}} M(dx)
 \ge \GfM{B}\,,
\end{equation*}
as $M\in\MeasuresShearerExists$ implies that the integrand is non-negative.  If $k>1$, then choose $\EnumSet{A_i}_{i=0}^k$ with $A=:A_0\subsetneq A_1\subsetneq\dotso\subsetneq A_k:=B$ and $\UPartNumOf{A_i\setminus A_{i-1}}=1$, for all $i\in\IntRange{k}$. Apply the statement for the previous case $k$ times and obtain
\begin{equation*}
 \GfM{A}=\GfM{A_0}\ge\GfM{A_1}\ge\dotso\ge\GfM{A_k}=\GfM{B}
 \,.\qedhere
\end{equation*}
\end{proof}

\begin{Prop}
\label{prop_gfr_boundedByOne_finite}
If $A,B\in\BorelFinite$ with $\GfM{B}>0$, then
\begin{equation}
\label{eq_fer_boundedByOne_finite}
 \GfMR{A}{B}\le 1\,.
\end{equation}
\end{Prop}

\begin{proof}
Without loss of generality, restrict to the case of $A\in\BorelUnitDiam$, telescoping~\eqref{eq_telescope} otherwise. Assume that $A\cap B=\emptyset$. Apply the FE~\eqref{eq_fe} to get
\begin{equation*}
 \GfMR{A}{B}
 = 1-\int_{A} \GfMR{B}{B\setminus\UnitSphere{x}}^{-1} M(dx)
 \le 1\,,
\end{equation*}
where the inequality holds if the integrand is positive. By assumption, $\GfM{B}>0$. Thus, for every $x\in A$, also $\GfM{B\setminus\UnitSphere{x}}>0$ by~\eqref{eq_fe_monotonicity_space} and the integrand $\GfMR{B}{B\setminus\UnitSphere{x}}^{-1}$ is positive.
\end{proof}

\begin{Prop}
\label{prop_gfr_monotonicity_space}
For each $A,A',B,B'\in\BorelFinite$, if
\begin{equation}
\label{eq_fer_monotonicity_space}
 A\subseteq A',
 B\subseteq B',
 A\cap B=A\cap B'
 \Then
 \GfMR{A}{B}\ge\GfMR{A'}{B'}\,.
\end{equation}
\end{Prop}

\begin{proof}
The general case combines the following two inequalities, varying in only one of the two arguments:
\begin{equation*}
 \GfMR{A}{B}
 \ge\GfMR{A'}{B}
 \ge\GfMR{A'}{B'}\,.
\end{equation*}

Case $B=B'$: Factorise and apply the upper bound with value $1$ from~\eqref{eq_fer_monotonicity_space} to the first factor to get
\begin{equation*}
 \GfMR{A'}{B}
 =\GfMR{A'}{A\cup B}\GfMR{A}{B}
 \le \GfMR{A}{B}\,.
\end{equation*}

Case $A=A'$: Use induction over $k:=\UPartNumOf{A\cup B'}$. If $k=0$, then $A=B=B'=\emptyset$ and $\GfMR{\emptyset}{\emptyset}=1$. If $k>0$, telescope~\eqref{eq_telescope} to restrict ourselves to the case $A\in\BorelUnitDiam$ and $A\cap B=A\cap B'=\emptyset$. Let $(A'_i)_{i=1}^k$ be a $\BorelUnitDiam$-partition of $A\sqcup B'$. Derive a $\BorelUnitDiam$-partition $(A_i)_{i=1}^k$ of $A\sqcup B$ via $A_i:=A'_i\cap(A\cup B)$. For $x\in A$, let $A'(x)$ and $A(x)$ be the unique partition elements containing $x$.\\

For $x\in A, y\in A(x)\subseteq A'(x)$, $A(x)\subseteq A'(x)\subseteq\UnitSphere{y}$ and $\UPartNumOf{B\setminus A(x)}\le\UPartNumOf{B'\setminus A'(x)}\le k-1$. Apply the first case and the inductive hypothesis to obtain
\begin{equation*}
 \GfMR{B\setminus A(x)}{B\setminus\UnitSphere{y}}
 \ge\GfMR{B'\setminus A'(x)}{B\setminus\UnitSphere{y}}
 \ge\GfMR{B'\setminus A'(x)}{B'\setminus\UnitSphere{y}}
 \,.
\end{equation*}
Apply the preceding inequality, twice the FE~\eqref{eq_fe} and enlarge the integration domain to get
\begin{align*}
 \GfMR{B\cap A(x)}{B\setminus A(x)}
 ={}&1-\int_{B\cap A(x)}
  \GfMR{B\setminus A(x)}{B\setminus\UnitSphere{y}}^{-1} M(dy)
 \\\ge{}& 1 - \int_{B\cap A(x)}
  \GfMR{B'\setminus A'(x)}{B'\setminus\UnitSphere{y}}^{-1} M(dy)
 \\\ge{}& 1 - \int_{B'\cap A'(x)}
  \GfMR{B'\setminus A'(x)}{B'\setminus\UnitSphere{y}}^{-1} M(dy)
 \\={}&\GfMR{B'\cap A'(x)}{B'\setminus A'(x)}
 \,.
\end{align*}
Again, for each $x\in A\setminus B'$, the preceding inequality for the first factor and the inductive hypothesis for the second factor imply that
\begin{align*}
 \GfMR{B}{B\setminus\UnitSphere{x}}
 ={}&
  \GfMR{B\cap A(x)}{B\setminus A(x)}
  \GfMR{B\setminus A(x)}{B\setminus\UnitSphere{x}}
 \\\ge{}&
  \GfMR{B'\cap A'(x)}{B'\setminus A'(x)}
  \GfMR{B'\setminus A'(x)}{B'\setminus\UnitSphere{x}}
 \\={}&\GfMR{B'}{B'\setminus\UnitSphere{x}}
 \,.
\end{align*}
Another two applications of the FE~\eqref{eq_fe} and the preceding inequality yield
\begin{multline*}
 \GfMR{A}{B}
 = 1-\int_{A}
  \GfMR{B}{B\setminus\UnitSphere{x}}^{-1} M(dx)
 \\\ge 1-\int_{A}
  \GfMR{B'}{B'\setminus\UnitSphere{x}}^{-1} M(dx)
 = \GfMR{A}{B'}
 \,.\qedhere
\end{multline*}
\end{proof}

\section*{Acknowledgements}
\Acknowledgements{}

\bibliographystyle{plain}
\bibliography{ref}

\end{document}